\documentclass[11pt]{article}

\usepackage[T1]{fontenc}
\usepackage[utf8]{inputenc}
\usepackage[a4paper,margin=1in]{geometry}
\usepackage{amsmath}
\usepackage{amssymb}
\usepackage{amsthm}
\usepackage{mathrsfs}
\usepackage{natbib}
\usepackage{graphicx}
\usepackage{xcolor}
\usepackage{hyperref}
\hypersetup{
	colorlinks=true,
	linkcolor=blue,
	citecolor=blue,
	urlcolor=blue
}

\makeatletter
\def\namedlabel#1#2{\begingroup
#2%
\def\@currentlabel{#2}%
\phantomsection\label{#1}\endgroup}
\makeatother

\theoremstyle{plain}

\newtheorem{theorem}{Theorem}[section]

\newtheorem{proposition}[theorem]{Proposition}
\newtheorem{remark}[theorem]{Remark}
\newtheorem{corollary}[theorem]{Corollary}
\theoremstyle{definition}
\newtheorem{definition}[theorem]{Definition}

\newcommand{\R}{{\mathbb R}}
\newcommand{\N}{{\mathbb N}}
\newcommand{\pfc}{{\mathcal F_c}}
\newcommand{\pkc}{{\mathcal K_c}}
\newcommand{\eps}{{\varepsilon}}
\newcommand{\del}{{\delta}}

\DeclareMathOperator{\midd}{mid}
\DeclareMathOperator{\spr}{spr}

\newcommand{\keywords}[1]{\vspace{0.5em}\noindent\textbf{\textit{Keywords:}} #1}

\title{On the continuity of the Tukey depth function for fuzzy data}

\author{
Luis Gonz\'alez-de-la-Fuente\thanks{Departamento de Matem\'aticas, Estad\'istica y Computaci\'on, Universidad de Cantabria. Avda.\ Los Castros s/n, 39005 Santander, Spain. \texttt{gdelafuentel@unican.es}}
\and
Alicia Nieto-Reyes\thanks{Departamento de Matem\'aticas, Estad\'istica y Computaci\'on, Universidad de Cantabria. Avda.\ Los Castros s/n, 39005 Santander, Spain. \texttt{alicia.nieto@unican.es}}
\and
Pedro Ter\'an\thanks{Departamento de Estad\'istica e Investigaci\'on Operativa y Did\'actica de las Matem\'aticas, Universidad de Oviedo. C/ Leopoldo Calvo Sotelo n\textordmasculine\ 18, 33007 Oviedo, Spain. \texttt{teranpedro@uniovi.es}}
}

\date{}

\begin{document}

\maketitle

\begin{abstract}
The practical use of statistical depth for fuzzy data requires regularity, inferential stability, and computational feasibility. This paper studies these aspects for the first time for the fuzzy depth, in particular for the fuzzy Tukey depth. In particular, we investigate the continuity of this depth with respect to each of its arguments. As a function of the elements of the underlying space, we prove its upper semicontinuity  under the main metrics used in fuzzy spaces. As a function of the distribution with respect to which the depth is computed, we establish the almost sure uniform consistency of its empirical version for continuous fuzzy random variables. These results ensure closed depth regions and support the use of sample depth values to approximate their population counterparts. We also derive consistency of empirical deepest-point estimators and study a finite-grid approximation of the depth, supported by theoretical and simulation results.
\end{abstract}


\bigskip

\section{Introduction}

The extension of order-based statistical methods beyond the real line is the central motivation behind the theory of statistical depth \citep{Serflingcondiciones}. In spaces within the real line, the inherent natural order allows the definition of medians, quantiles, and ranks in a direct and geometrically transparent way. In multivariate, and non-Euclidean spaces, this total order is no longer available \citep{trianguloCabrera}. Statistical depth functions overcome this difficulty by assigning to each element of the underlying space a numerical value representing its degree of centrality with respect to a probability distribution on the space. A depth function therefore induces a centre-outward ordering of the space with respect to that distribution, providing a basis for robust and nonparametric statistical procedures  \citep{AgostinelliNietoFrancisciClustering, GnettnerKirchTest}.

The half-space, or Tukey, depth function was introduced by \citet{tukey} in the multivariate setting as a way of measuring  the centrality of a point within a cloud of observations through the use of halfspaces. Thus, the multivariate halfspace depth of $x\in\mathbb R^p$ with respect to a p-dimensional random vector $X$ associated with a probability space $(O,\sigma, P)$ is 
\begin{equation}\label{TukeyR}
	HD(x;X) := \inf\{
	P_X(H) : H\text{ is a closed halfspace with }x\in H\}.
\end{equation}
Despite Tukey depth remains as the most prominent example of depth in the literature, because of its good theoretical properties, a substantial number of multivariate depth functions have been proposed since then; these include  Oja depth  \citep{Oja}, simplicial depth \citep{Liu}, regression depth \citep{regressiondepth} and spatial depth \citep{spatialdepth}.  Additionally, a systematic axiomatic treatment for the notion of multivariate  statistical depth was later provided in \citet{ZuoSerfling}, placing depth functions within a general framework for multivariate ordering. This resulted in that not all the existing examples of depth satisfy the notion.

The development of statistical depth has not been restricted to multivariate spaces. Modern statistical problems frequently involve observations that are not vectors in \(\mathbb{R}^p\), but functions, sets, shapes, distributions or other complex objects. This has motivated the construction of depth notions in functional and more general data spaces. In the functional setting, the literature contains several examples of depth, such as the \(h\)-depth \citep{cuevashdepth}, the band and modified band depth  \citep{lopezpintado} and metric based depths \citep{HeatherJMVA, math, depthmetric}.  
Some of these examples were introduced before the formulation of general axiomatic framework; \citet{NietoBattey} proposed an axiomatic definition of statistical depth for functional metric spaces, adapting the multivariate principles  and adding others to suit the particular structure of functional spaces.

Fuzzy data constitute a particularly relevant class of nonstandard data. They arise when observations are affected not only by randomness but also by imprecision, vagueness or partial membership. Since the seminal work of \citet{zadehfuzzysets}, fuzzy sets have provided a flexible mathematical framework for representing objects whose boundaries are not sharply defined. In statistical applications, fuzzy data appear naturally in contexts involving expert assessments \citep{coppigilkiers}, linguistic variables \citep{zadehextension}, imprecise measurements \citep{krusemeyer}, and interval-like information \citep{mirandacousogil}; and common statistical methods are defined specifically for the fuzzy framework, such as hypothesis testing  \citep{testfuzzy,lubianotest,montenegrotest}.
The first general notions of statistical depth for fuzzy data were proposed by \citet{primerarticulo}, who introduced two axiomatic approaches depending on whether the fuzzy space is interpreted through its algebraic structure or through its metric structure. Within this work, the Tukey depth for fuzzy sets appears as a natural extension of the classical halfspace depth and is proved to satisfy the two axiomatic notions. 
Other generalizations of depth functions to the fuzzy framework have followed, such as generalizations of the simplicial depth \citep{tercerarticulo} and  of the projection depth and the $L^r$-type depths  \citep{sinovadepth, cuartoarticulo}.

{\bf Motivation:} Once a depth function has been defined, its usefulness depends not only on its 
satisfaction of the properties constituting the notion of depth on the corresponding space, but also on its regularity, inferential stability and computability. In this respect, as depth functions are evaluated on an element of the underlying space, say $x\in\mathcal{F},$ with respect to a distribution on a family of distributions, say $\mathcal{P}\in\mathscr{P},$ defined over $\mathcal{F},$ i.e.
$$
\begin{array}{lcll}
D(\cdot;\cdot):& \mathcal{F} \times \mathscr{P}&\longrightarrow &\mathbb{R}\\
&(x;\mathcal{P})&\longmapsto &D(x;\mathcal{P}),
\end{array}
$$
\citet{Serflingcondiciones} proposed to study the continuity of depth functions with respect to each of its arguments,
\begin{itemize}
\item continuity of $D(x;\mathcal{P})$ as a function of $x$, or merely upper semi-continuity
\item continuity of $D(x;\mathcal{P})$ as a function of $\mathcal{P}$, or merely uniform consistency,
\end{itemize}
as desirable properties that a depth should fulfill in order to behave appropriately.

Upper-semicontinuity is a natural regularity requirement for a depth function, which  is essential for the study of depth contours and deepest elements, median-type objects. 
Upper-semicontinuity ensures that the depth value remains stable when the point at which the depth is calculated is perturbed slightly, preventing upward jumps of the depth under small perturbations. In fuzzy spaces, where observations may be approximated through discretized membership functions or finite collections of \(\alpha\)-levels, such stability is especially relevant. Upper-semicontinuity also guarantees the closeness of every depth inner region, which is a valuable tool when working with depth functions \citep{zuoserflingcontours}. In the fuzzy framework, the closeness is with respect to a given topology induced by a chosen metric on the fuzzy space. 
Uniform consistency is the corresponding inferential property. As with any theoretical depth function, the population fuzzy Tukey depth depends on the distribution of the underlying fuzzy distribution, which is unknown in applications. Hence, it must be approximated by its empirical version. Pointwise consistency is useful, but insufficient for many depth-based procedures, since rankings, central regions, medians and outlier detection depend on the depth function over a whole class of fuzzy sets. Uniform consistency of the sample version to the population depth provides the stronger guarantee that the empirical depth approximates the population depth globally, providing theoretical justification for approximating the true depth values with those obtained from the sample. 

\citet{ZuoSerfling} proved the upper-semicontinuity property for a general family of multivariate depths, which includes Tukey depth, and \citet{Donoho92} demonstrated that the sample version of the, multivariate, Tukey depth is consistent with respect to its population counterpart.
 In addition to the multivariate Tukey depth, examples of multivariate depth functions that satisfy these properties are the projection depth \citep{ZuoProjection} and the simplicial depth \citep{Liu,Dumbgen}. 
Furthermore,  the random Tukey depth, multivariate \citep{randomTukey} and functional \citep{rfTd},  is  consistent with respect to its population counterpart. 
 The importance of these  properties is such that they are axioms in the notion of statistical depth in functional spaces. Examples of functional depth that satisfy these properties are the band depth \citep{lopezpintado} and the half-region depth \citep{lopezpintadohalf}. These properties have not yet been studied, though, in the fuzzy setting for any depth function.  

{\bf Contributions:} The main contributions of this paper are to establish upper-semicontinuity and uniform consistency of the Tukey depth for fuzzy data. 
 In proving upper-semicontinuity of the fuzzy Tukey depth, we show that it results in  controlling the behaviour of the 
 associated compact \(\alpha\)-levels under convergence in the relevant fuzzy metrics. 
To prove that the Tukey depth for fuzzy sets is uniformly consistent with respect to its sample version, we prove the supremum over all the fuzzy sets of the absolute difference between the population and the sample version converges to $0$ almost surely,{ whenever the  considered theoretical fuzzy distribution is continuous.}

Furthermore, as a consequence of these results, we study the asymptotic behaviour of the empirical 1-median of a fuzzy distribution. Taking into account that the univariate median is the minimizer of the expectation of the difference of an absolute value, \citet{medianfuzzy1} introduced the concept of $1$-median of a fuzzy distribution as the fuzzy sets that minimize the expectation of a fuzzy distance based on  the $L^1$ distance. In \cite{quintoarticulo} it is proved that the Tukey depth for fuzzy sets is maximized at $1$-medians. Combining this connection with the upper-semicontinuity and uniform consistency results obtained here, we derive the convergence of empirical median-type estimators. In particular, when the theoretical fuzzy distribution is continuous,  the 1-median is unique, and the distance between the empirical and population 1-medians converges to zero in a fuzzy metric.

The theoretical analysis is complemented by a computational component. As for the, multivariate, Tukey depth \citep{randomTukey}, exact computation of fuzzy Tukey depth is generally infeasible, since its definition involves an infimum over the unit sphere on $\mathbb R^p$ and the interval [0,1]. 
The infimum over the unit sphere on $\mathbb R^p$ can be solved following the ideas of the random Tukey depth in \citet{randomTukey}.
However, even in the univariate fuzzy case, where the unit sphere is finite, the continuum of the interval [0,1] remains. Therefore, practical implementation requires a finite approximation. We propose to approximate the depth by replacing the interval \([0,1]\) with a finite grid. 
We prove that, for the common case of triangular fuzzy numbers, this approximation converges to the exact fuzzy Tukey depth as the grid becomes finer.

{We include a  simulation study that demonstrates the practical implications of the theoretical results. We consider triangular fuzzy random variables for which the theoretical 1-median of the underlying distribution is either available in closed form or can be computed numerically. In the first setting, the known population 1-median serves as a benchmark for comparison with the fuzzy observation in the sample that attains the highest empirical approximated Tukey depth. In the second setting, the empirical fuzzy observation with maximal depth is regarded as a representative central element of the sample. The simulations illustrate how the empirical approximated depth evolves as the sample size increases and how the theoretical convergence results are reflected in finite-sample behavior.}

{\bf Organization of the paper.} Section~\ref{Preliminaries} introduces the notation and preliminary material on fuzzy sets. 
Section~\ref{main} proves the upper-semicontinuity of the fuzzy Tukey depth, which results in the closeness of the depth $\alpha$-levels, and establishes the almost sure uniform consistency of its empirical version. The same section also derives the asymptotic convergence of the empirical 1-median.  Section~\ref{pf} includes the proofs of the results in the previous section.
Section~\ref{compt} is dedicated to the computability of the fuzzy Tukey depth. It proposes a computational approximation of the depth and studies its consistency, theoretically and through a simulation study. The proofs of the results in Section~\ref{compt} and further simulations are deferred to the Appendix. 
Finally, Section~\ref{conclusions} contains concluding remarks and discusses possible extensions.

	\section{Preliminaries, on fuzzy sets}\label{Preliminaries}
This section provides the necessary definitions, notation and results that will be used throughout the paper.
A fuzzy set $A$ on $\mathbb R^p$ is defined as a mapping $A: \mathbb{R}^{p}\rightarrow [0,1]$, where for each $x\in\mathbb R^p$, $A(x)$ represents the membership degree of $x$ with respect to $A$. Given any $\alpha\in (0,1]$, the $\alpha$-level (or $\alpha$-cut) of a fuzzy set $A$, denoted by $A_\alpha$, is the subset of $\mathbb R^p$ whose elements have membership degree greater or equal than $\alpha$, that is, the set $$A_\alpha := \{x\in\mathbb R^p : A(x)\geq\alpha\}.$$ The $0$-level, denoted by $A_0$, is defined as the closure of the set of elements in $\mathbb R^p$ whose membership degree is strictly greater than $0$, that is the set $$A_0 := \text{clo}(\{x\in\mathbb R^p : A(x)>0\}),$$ where $\text{clo}(\cdot)$ denotes the closure of a set with respect to the usual topology in $\mathbb R^p$. We denote by $\mathcal K_c(\mathbb R^p)$ the class of compact, convex and non-empty subsets of $\mathbb R^p$, by $\mathcal{F}(\mathbb{R}^{p})$ the set of fuzzy sets on $\mathbb{R}^{p}$ and by $\mathcal{F}_{c}(\mathbb{R}^{p})$ the set of fuzzy sets on $\mathbb R^p$, $A$, with $\alpha$-levels compact convex and non-empty, that is, $A_\alpha\in\mathcal K_c(\mathbb R^p)$ for every $\alpha\in [0,1]$. This paper focuses  on $\mathcal{F}_{c}(\mathbb{R}^{p})$. For simplicity, we refer to the elements of $\mathcal{F}_{c}(\mathbb{R}^{p})$ as fuzzy sets, although this is a sligth abuse of nomenclature.

If we identify each subset $L\subseteq\mathbb R^p$ with its indicator function, $I_L:\mathbb R^p\rightarrow\{0,1\}$, defined as $I_L(x) = 1$ if $x\in L$ and $I_L(x) = 0$ if $x\not\in L$, then any subset of $\mathbb R^p$ can be viewed as a fuzzy set. Thus, $L$ can be regarded as an element of $\mathcal F(\mathbb R^p)$ via its indicator function, $I_L$. Furthermore, every $K\in\mathcal K_c(\mathbb R^p)$ can be interpreted as an element of $\mathcal F_c(\mathbb R^p)$. Thus, abusing of the terminology, we assume that $\mathcal{K}_{c}(\mathbb{R})\subseteq\mathcal{F}_{c}(\mathbb{R})$.
The unit sphere on $\mathbb R^p$ is denoted by $\mathbb{S}^{p-1} := \{x\in\mathbb{R}^{p}: \|x\|\leq 1 \}$, where $\|.\|$ is the euclidean norm on $\mathbb{R}^{p}$.
The support function of $K\in\mathcal{K}_{c}(\mathbb{R}^{p})$ is defined as a mapping $s_{K}: \mathbb{S}^{p-1}\rightarrow\mathbb{R}$ with $$s_{K}(u) := \sup\{\langle u,k\rangle : k\in K\}\mbox{
for every }u\in\mathbb{S}^{p-1},$$ where $\langle \cdot ,\cdot\rangle$ denotes the usual inner product in $\mathbb{R}^{p}$. Given a fuzzy set $A$, the support function of $A$ is the mapping $s_A:\mathbb S^{p-1}\times [0,1]\rightarrow\mathbb R$ defined by 
\begin{equation*}
s_A(u,\alpha) =: \sup\{\langle u,a\rangle: a\in A_\alpha\}. 
\end{equation*}

Let us denote by $(\Omega,\mathcal{A},\mathbb{P})$ a probability space. According to \citet{Mol} a compact and convex random set is a function $$\Gamma:\Omega\rightarrow\mathcal{K}_{c}(\mathbb{R}^{p})\mbox{ such that }\{\omega\in\Omega : \Gamma(\omega)\cap K\neq\emptyset \}\in\mathcal{A}\mbox{ for all }K\in\mathcal{K}_{c}(\mathbb{R}^{p}).$$  \citet{PuriRalescu} defined a fuzzy random variable as a function $\mathcal{X}:\Omega\rightarrow\mathcal{F}_{c}(\mathbb{R}^{p})$ such that $\mathcal{X}_{\alpha}(\omega)$ is a random compact set for all $\alpha\in[0,1]$, where $\mathcal{X}_{\alpha}:\Omega\rightarrow\mathcal{P}(\mathbb{R}^{p})$ is defined as $$\mathcal{X}_{\alpha}(\omega) := \{x\in\mathbb{R}^{p}: \mathcal{X}(\omega)(x)\geq\alpha  \}\mbox{ for any }\omega\in\Omega.$$
It is not explicit in the above definition of fuzzy random variable that it is a measurable function in the ordinary sense. It is easy to see that $\mathcal{X}$ is a fuzzy random variable if and only if it is measurable when  the fuzzy space $\mathcal{F}_{c}(\mathbb{R}^{p})$ is endowed with the $\sigma$-algebra generated by the $\alpha$-cut mappings $$L_\alpha:A\in\mathcal{F}_{c}(\mathbb{R}^{p})\mapsto A_\alpha\in\mathcal{K}_{c}(\mathbb{R}^{p}).$$ This is the smallest $\sigma$-algebra which makes each mapping $L_\alpha$ measurable. \citet{kra} showed that this $\sigma$-algebra is the Borel $\sigma$-algebra generated by any metric of the families $d_r$ or $\rho_r$ for $r\in[1,\infty)$.
Finally, given a fuzzy random variable $\mathcal{X}: \Omega\rightarrow\mathcal{F}_{c}(\mathbb{R}^{p})$, the support function of $\mathcal X$, denoted by $s_{\mathcal{X}}$, is a function $s_{\mathcal{X}} : \mathbb{S}^{p-1}\times [0,1]\times\Omega\rightarrow\mathbb{R}$ defined by
\begin{equation}\label{srv}
	s_{\mathcal{X}}(u,\alpha,\omega) := s_{\mathcal{X}(\omega)}(u,\alpha), \mbox{ for all }u\in\mathbb{S}^{p-1}, \alpha\in [0,1]\mbox{ and }\omega\in\Omega.
\end{equation}
For every $u\in\mathbb{S}^{p-1}$ and $\alpha\in [0,1]$, by definition of fuzzy random variable, $s_{\mathcal{X}}(u,\alpha)$ is a real random variable.
%
%
%
%

We denote by $L^{0}[\mathcal{F}_{c}(\mathbb{R}^{p})]$ the class of all fuzzy random variables on the measurable space $(\Omega,\mathcal{A}).$ 
Additionally, 
\begin{equation*}C^{0}[\mathcal{F}_{c}(\mathbb{R}^{p})]:=\{\mathcal{X}\in L^{0}[\mathcal{F}_{c}(\mathbb{R}^{p})]: s_{\mathcal{X}}(u,\alpha) \mbox{ is  continuous  for each } (u,\alpha)\in\mathbb{S}^{p-1}\times [0,1]\}\end{equation*}
and
$L^1[\pfc(\R^p)]:=\{\mathcal X\in L^{0}[\mathcal{F}_{c}(\mathbb{R}^{p})]: E[\|\mathcal X\|]<\infty\}.$
 We refer to the fuzzy random variables in  $C^{0}[\mathcal{F}_{c}(\mathbb{R}^{p})]$ as continuous and to those in
$L^1[\pfc(\R^p)]$ as  integrably bounded.

	\section{Main results}\label{main}

The Tukey depth for fuzzy sets \citep{primerarticulo}, or fuzzy Tukey depth,  is a generalization of the multivariate halfspace depth, in Equation \eqref{TukeyR}, which is based on the support functions of fuzzy sets. Thus, given the underlying probability space $(\Omega,\mathcal{A},\mathbb{P})$, the Tukey depth, $D_{FT}$, based on $\mathcal{H}\subset L^{0}[\mathcal{F}_{c}(\mathbb{R}^{p})]$ and $\mathcal{J}\subset\mathcal{F}_{c}(\mathbb{R}^{p})$ of a fuzzy set $A\in\mathcal{J}$ with respect to $\mathcal{X}\in\mathcal{H}$ is defined as
\begin{equation}\label{TukeyF}
D_{FT}(A;\mathcal{X}) := \inf_{u\in\mathbb{S}^{p-1}, \alpha\in [0,1]} \min(\mathbb{P}[\omega\in\Omega: \mathcal{X}(\omega)\in S_{u,\alpha}^{-}],\mathbb{P}[\omega\in\Omega: \mathcal{X}(\omega)\in S_{u,\alpha}^{+}]),
\end{equation}
where
\begin{equation}
	\begin{aligned}\nonumber
		S_{u,\alpha}^{-} &= \{U\in\mathcal{F}_{c}(\mathbb{R}^{p}) : s_{U}(u,\alpha) - s_{A}(u,\alpha)\leq 0\}\\ \nonumber
		S_{u,\alpha}^{+} &= \{U\in\mathcal{F}_{c}(\mathbb{R}^{p}) : s_{U}(u,\alpha) - s_{A}(u,\alpha)\geq 0\}.
	\end{aligned}
\end{equation}
	In this section we study the continuity of the fuzzy Tukey depth, in addition to results on fuzzy median and the computability of the fuzzy Tukey depth.
	
	\subsection{Upper-semicontinuity} 

As commented in the introduction, it is important to consider if depth functions are continuous functions on the first argument 
or, at least, if they are upper semicontinuous functions. Upper-semicontinuity was introduced in \citet{Serflingcondiciones} for multivariate depths with the following definition. A function $D:\mathbb R^p\rightarrow\mathbb R$ is upper-semicontinuous if for all $x\in\mathbb R^p,$ it happens that
	\begin{equation}\nonumber
		\limsup_{\|x-y\|\rightarrow 0}D(y)\leq D(x).
	\end{equation}
	Upper-semicontinuity of the multivariate Tukey depth is proved in \citet{Donoho92}.
	While the property is not studied for the Tukey depth in  functional spaces, because the depth is degenerate in such spaces \citep{indios},  it is studied for compact and convex 
 	sets.
 For compact and convex 
  sets, upper-semicontinuous is defined with respect to a metric. Given a function $D:\mathcal K_c(\mathbb R^p)\rightarrow \mathbb R$ and a metric $d$  on $\mathcal K_c(\mathbb R^p)$, we say that $D$ is upper-semicontinuous if for all $K\in\mathcal K_c(\mathbb R^p)$ and for any sequence $\{K_n\}_n\subset\mathcal K_c(\mathbb R^p)$ such that $\lim_n d(K,K_n) = 0$, we have that 
\begin{equation}\nonumber
	\limsup_n D(K_n)\leq D(K).
\end{equation}
A common metric to use is the Hausdorff metric, $d_H : {\mathcal K_c(\mathbb R^p)\times \mathcal K_c(\mathbb R^p)}\rightarrow [0,\infty),$ defined by
\begin{equation}\label{hm}
	d_H (K,T) := \max\left\{\sup_{k\in K}\inf_{t\in T} \parallel k-t\parallel, \sup_{t\in T}\inf_{k\in K}\parallel k-t\parallel\right\}
\end{equation}
for any $K,T{\in\mathcal K_c(\mathbb{R}^{p})}$.

 Given the underlying probability space $(\Omega,\mathcal A,\mathbb P)$, 
the Tukey depth, $D_{CT}$, of $K{\in\mathcal K_c(\mathbb{R}^{p})}$ with respect to a compact and convex random set $\Gamma:\Omega\rightarrow\mathcal K_c(\mathbb R^p)$ is defined in \cite{segundoarticulo} as the function
	\begin{equation}\label{CT}
		D_{CT}(K;\Gamma) := \min\{\inf_{u\in\mathbb{S}^{p-1}}\mathbb{P}(s_{\Gamma}(u)\leq s_{K}(u)),\inf_{u\in\mathbb{S}^{p-1}}\mathbb{P}(s_{\Gamma}(u)\geq s_{K}(u))\}.
	\end{equation}
The proof of the main result in this subsection, Theorem \ref{TukeyUpperSemi}, 	makes use of the upper-semicontinuity of the Tukey depth for compact and convex random sets, which we include in Proposition \ref{propositionTukeyCompact} below for the sake of completeness. 
	\begin{proposition}[\citet{segundoarticulo}]\label{propositionTukeyCompact}
For any compact and convex random set $\Gamma,$ the depth function $D_{CT}(\cdot;\Gamma)$ is upper-semicontinuous with respect to the $d_H-$metric.
\end{proposition}

Next, we define upper-semicontinuity for fuzzy sets, following the definition for compact and convex random sets. As the property involves a metric, we consider it here in relation to the notion of metric fuzzy depth function. 
\begin{definition}\label{USF}
Given a metric $d$ defined over $\mathcal F_c(\mathbb R^p)$ and a function $D:\mathcal F_c(\mathbb R^p)\rightarrow\mathbb R$, we say that $D$ is \emph{upper-semicontinuous} if for all  $A\in\mathcal F_c(\mathbb R^p)$ and for any sequence $\{A_n\}_n\subseteq\mathcal F_c(\mathbb R^p)$ such that $\lim_n d(A,A_n) = 0$, we have that 
\begin{equation*}
\limsup_n D(A_n)\geq D(A).
\end{equation*}
\end{definition}
%
%
%
%
%

A well-known family of metrics  over $\mathcal{F}_{c}(\mathbb{R}^{p})$ is based on the Hausdorff metric over the $\alpha$-levels, in Equation \eqref{hm}, and we refer to it as 
the $d_{r}$ family. 
Given two fuzzy sets $A,B\in\mathcal{F}_{c}(\mathbb{R}^{p})$, the family of metrics $d_r$, with $r\in [1,\infty]$, is
\begin{equation*}
d_{r}(A,B) := \left\{ \begin{array}{lcc}
	\left(\int_{0}^{1} \left( d_H (A_\alpha,B_\alpha) \right)^{r} d\nu(\alpha)\right)^{1/r}, &  r\in [1,\infty)\\
	\\ \sup_{\alpha\in [0,1]} d_H  (A_\alpha,B_\alpha), &   r = \infty
\end{array}
\right.
\end{equation*}
where $\nu$ denotes the Lebesgue measure over the closed interval $[0,1].$ 
\citet{diamondkloden} demonstrated that  $(\mathcal{F}_{c}(\mathbb{R}^{p}), d_{r})$ is a non-complete and separable metric space for any $r\in(1,\infty)$, while $(\mathcal{F}_{c}(\mathbb{R}^{p}), d_{\infty})$ is a non-separable and complete metric space.

Next result states the upper-semicontinuity of the Tukey depth for fuzzy sets with respect to the $d_1$ metric. 
\begin{theorem}\label{TukeyUpperSemi}
	For any $\mathcal X\in L^0[\mathcal F_c(\mathbb R^p)]$, the depth function $D_{FT}(\cdot;\mathcal X)$ is upper semicontinuous with respect to the $d_1$-metric. 
\end{theorem}
The proofs of the results in this section are in Section \ref{pf}.

Another well-known family of fuzzy metrics is the family of  $L^{r}$-type metrics \citep{diamondkloden}.
\begin{equation*}
	\rho_{r}(A,B) := \left(\int_{\mathbb{S}^{p-1}}\int_{[0,1]}\|s_{A}(u,\alpha)-s_{B}(u,\alpha)\|^{r}d\nu(\alpha)\, d\mathcal V_p(u)\right)^{1/r}
\end{equation*}
for every $A,B\in\mathcal{F}_{c}(\mathbb{R}^{p})$ with $r\in [1,\infty)$, where $\mathcal V_p$ denotes the uniform normalized distribution over the unit sphere $\mathbb S^{p-1}$.

Since  upper-semicontinuity in Definition \ref{USF} is a property that depends on the topology induced by the metric on the space $\mathcal F_c(\mathbb R^p)$, using Theorem \ref{TukeyUpperSemi} we prove below upper-semicontinuity for $D_{FT}$ with respect to $d_\infty$ and $d_r$ and $\rho_r$ for any $r\in [1,\infty).$ 
	
	\begin{corollary}\label{corolarioupper}
		For any $\mathcal X\in L^0[\mathcal F_c(\mathbb R^p)]$, the depth function $D_{FT}(\cdot;\mathcal X)$ is upper semicontinuous with respect to the metric $d_\infty$ and with respect to $d_r$ and $\rho_r$ metrics, for any $r\in [1,\infty).$
	\end{corollary}
	
	Upper-semicontinuity implies that the $\alpha$-levels of $D_{FT}$ are closed subsets of $\mathcal F_c(\mathbb R^p)$ for the correponding metrics.
		\begin{corollary}
			Let $\mathcal X\in L^0[\mathcal F_c(\mathbb R^p)]$ be a fuzzy random varible and $\alpha\geq 0$. The $\alpha$-level of $D_{FT}(\cdot;\mathcal X)$, defined as
			\begin{equation}\nonumber
				\{A\in\mathcal F_c(\mathbb R^p) : D_{FT}(A;\mathcal X)\geq\alpha\},
			\end{equation}
			is a closed set with respect to the metric $d_\infty$ and with respect to the metrics $d_r$ and $\rho_r$ for any $r\in [1,\infty)$.
		\end{corollary}

\subsection{Consistency of fuzzy depth instances }


Another desirable property that is studied in the statistical depth functions' literature is the consistency of the empirical depth. In particular, the uniform almost surely convergence between the empirical version of the depth and its population counterpart. 
Thus, in \citet{zuoserflingcontours} it is proved for the multivariate Tukey depth that
	\begin{equation}\nonumber
		\sup_{x\in\mathbb R^p}|HD_n(x;\{X_1,\ldots ,X_n\}) - HD(x;X)|\longrightarrow 0,\text{ almost surely }[P],
	\end{equation}
where $\{X_1,\ldots ,X_n\}$ is a set of i.i.d. random variables with distribution $X$ and 
\begin{equation}\label{HDn}
HD_n(x;\{X_1,\ldots ,X_n\}) := \inf\{ \mathbb P_n(H) : H\text{ is a closed halfspace with }x\in H \}
\end{equation} is the sample multivariate Tukey depth 
with $\mathbb P_n$  the empirical probability measure associated with  $X_1,\ldots ,X_n$.

To establish the property in the fuzzy framework, let $\mathcal{X}$ be a fuzzy random variable associated with a probabilistic space $(\Omega ,\mathcal{A} ,\mathbb{P})$ and $\mathcal{X}_{1},\cdots ,\mathcal{X}_{n}$  a simple random sample from $\mathcal{X}.$ 
 By \citet[Proposition 4.4]{medibilidadpedromiriam}, due to measurability arguments, we have that $s_{\mathcal{X}_{1}}(u,\alpha),\cdots ,s_{\mathcal{X}_{n}}(u,\alpha)$ is a simple random sample of $s_{\mathcal{X}}(u,\alpha)$ for all $u\in\mathbb{S}^{p-1}$ and $\alpha\in [0,1]$. In next definition, which presents the sample, or empirical, version of the fuzzy Tukey depth,  we make use of  the empirical distribution $\mathbb{P}_{u,\alpha}^{n}(A) := \sharp (A\cap\{s_{\mathcal{X}_{1}}(u,\alpha),\cdots ,s_{\mathcal{X}_{n}}(u,\alpha)\})/n$ for every $A\in\mathcal{A}$.



\begin{definition}
	The sample	 fuzzy Tukey depth, $D_{FT}^{n}$, based on $\mathcal{H}\subset L^{0}[\mathcal{F}_{c}(\mathbb{R}^p)]$  and $\mathcal{J}\subset\mathcal{F}_{c}(\mathbb{R}^p)$ of a fuzzy set $A\in\mathcal{J}$ with respect to  a simple random sample $\mathscr{X}:=\{\mathcal{X}_{1},\ldots ,\mathcal{X}_{n}\}$ from $\mathcal{X}\in\mathcal{H}$
	 is 
	$$D_{FT}^{n}(A;\mathscr{X}) := \min\{\inf_{u,\alpha}\mathbb{P}_{u,\alpha}^{n}((-\infty,s_A(u,\alpha)]), \inf_{u,\alpha}\mathbb{P}_{u,\alpha}^{n}([s_A(u,\alpha),\infty))  \}.$$
\end{definition}

In \cite{capituloTukey} it is proved that $D_{FT}$ can be express using the halfspace depth in $\mathbb{R}$, $HD$ in Equation \eqref{TukeyR}, as
\begin{equation}\label{ecuacion1Tukey}
	\begin{aligned}
		&D_{FT}(A;\mathcal{X}) = \inf_{u\in\mathbb S^{p-1}}\inf_{\alpha\in [0,1]} HD(s_{A}(u,\alpha);s_{\mathcal{X}}(u,\alpha)) =\\
		&\inf_{u\in\mathbb S^{p-1}}\inf_{\alpha\in [0,1]} \min\{\mathbb P(s_{\mathcal X}(u,\alpha)\leq s_A(u,\alpha)),\mathbb P(s_\mathcal X(u,\alpha)\geq s_A(u,\alpha)) \}
	\end{aligned}
\end{equation}
for every fuzzy random variable $\mathcal{X}\in L^{0}[\mathcal{F}_{c}(\mathbb{R}^{p})]$ and $A\in\mathcal{F}_{c}(\mathbb{R}^{p})$.
Following it, we also write the sample  fuzzy Tukey depth in terms of the sample halfspace depth in $\mathbb{R},$ in Equation \eqref{HDn}, 
as
\begin{equation}\label{ecuacion2Tukey}
	\begin{aligned}
		&D_{FT}^{n}(A;\mathscr{X}) = \inf_{u\in\mathbb S^{p-1}}\inf_{\alpha\in [0,1]} HD_{n}(s_{A}(u,\alpha); \{s_{\mathcal{X}_{1}}(u,\alpha),\cdots ,s_{\mathcal{X}_{n}}(u,\alpha)\}) =\\
		&\inf_{u\in\mathbb{S}^{p-1}}\inf_{\alpha\in [0,1]} \min\{\mathbb P_{u,\alpha}^n((-\infty,s_A(u,\alpha)]),\mathbb P_{u,\alpha}^n([s_A(u,\alpha),\infty)) \}.
	\end{aligned}
\end{equation}

We make use of these expressions to prove next result,  which  is the main result of uniform consistency for the sample fuzzy Tukey depth. We state the result for continuous fuzzy random variables.

\begin{theorem}\label{fuzzyTukeyconsistency}
	Given the underlying probability space $(\Omega,\mathcal{A},\mathbb{P})$, let $\mathcal{X}$ be a fuzzy random variable in $C^0[\pfc(\R^p)]$ and $\mathscr{X}:=\{\mathcal{X}_{1},\cdots ,\mathcal{X}_{n}\}$ be i.i.d. as $\mathcal{X}$. Then,
	$$\sup_{A\in\mathcal{F}_{c}(\mathbb{R}^{p})} |D_{FT}(A;\mathcal{X}) - D_{FT}^{n}(A;\mathscr{X})|\longrightarrow 0, \mbox{ almost surely } [\mathbb P].$$
\end{theorem}

\begin{remark}
	In the multivariate case, uniform consistency follows from the fact that halfspaces are a Vapnik--Chervonenkis class. In our framework, the analogs of halfspaces are not a Vapnik--Chervonenkis class, however. 
	For this purpose, we have developed a new technique to prove in Section \ref{pf} the uniform consistency for the empirical fuzzy Tukey depth with respect to its population counterpart.
\end{remark}


\subsection{Consistency of the fuzzy median}

Let  $\widetilde{\text{Med}}(\mathcal X)\in\mathcal F_c(\mathbb R)$ be defined by
$$
(\widetilde{\text{Med}}(\mathcal X))_\alpha := [\text{Med}(\inf\mathcal X_\alpha),\text{Med}(\sup\mathcal X_\alpha)]
$$ 
for every $\alpha\in [0,1]$, where $\text{Med}$ denotes the univaritate median, using the convention of taking the mid-point if the set of univariate medians of real-valued random variables is an interval. \citet{medianfuzzy1} proposed the notion of \textit{$1$-median} of a fuzzy random variable and proved that the fuzzy set $\widetilde{\text{Med}}(\mathcal X)\in\mathcal F_c(\mathbb R)$  is a $1$-median of the fuzzy random variable $\mathcal X$ \cite[Theorem 4.1]{medianfuzzy1}. 
A fuzzy set $A\in\mathcal F_c(\mathbb R)$ is a $1$-median of $\mathcal X$ if
$$
A\in\arg\min_{B\in\mathcal F_c(\mathbb R)}\text{E}[\rho_1(B,\mathcal X)],
$$
{
whenever this expectation exists.} 
Additionally, 
 \cite{quintoarticulo} proved that the set of $1$-medians of a fuzzy random variable $\mathcal X$ coincides with the set 
 that maximizes $D_{FT}(\cdot;\mathcal X)$. 
 
Following ideas in \citet{arcones}, next theorem proves that the sequence of empirical maximizers of $D_{FT}^n$ converges almost surely to the maximizer of $D_{FT}$ with {respect to the $d_\infty$ metric}, assuming that the maximizer of $D_{FT}$ is unique. 
Thus, the result shows that, if the 1-median of the corresponding fuzzy random variable is unique, then the empirical median converges (almost surely) to its population counterpart with {respect to the $d_\infty$ metric}. 
The subsequent corollary proves the convergence with respect to the $\rho_r$ and $d_r$ metrics for every $r\in [1,\infty).$ The result is for continuous fuzzy random variables in $\mathbb R^p.$
Thus,  our results generalize the consistency of the empirical $1$-median to the population $1$-median proved in \cite{medianfuzzy1} (only considered for $p = 1$ and the metric $\rho_1$).


\begin{theorem}\label{medianateorema} 
	Given the underlying probability space $(\Omega,\mathcal{A},\mathbb{P}),$ let $\mathcal X\in C^0[\mathcal F_c(\mathbb R^p)]$ be a fuzzy random variable such that $D_{FT}(\cdot;\mathcal X)$ is uniquely maximize in $M\in\mathcal F_c(\mathbb R^p)$. For each $n\in \mathbb N,$ let $\mathscr{X}:\{\mathcal{X}_{1},\cdots ,\mathcal{X}_{n}\}$ be i.i.d. as $\mathcal{X}$ and $M_n$ denote a maximizer of $D_{FT}^n(\cdot; \mathscr{X}).$ 
	Then,
	\begin{equation}\nonumber
		d_\infty(M,M_n)\longrightarrow 0, \text{ almost surely } [\mathbb{P}].
	\end{equation}
\end{theorem}


\begin{corollary}\label{medianacorolario} 
	Given the underlying probability space $(\Omega,\mathcal{A},\mathbb{P}),$ let $\mathcal X\in C^0[\mathcal F_c(\mathbb R^p)]$ be a fuzzy random variable such that $D_{FT}(\cdot;\mathcal X)$ is uniquely maximize in $M\in\mathcal F_c(\mathbb R^p)$. For each $n\in \mathbb N,$ let $\mathscr{X}:=\{\mathcal{X}_{1},\cdots ,\mathcal{X}_{n}\}$ be i.i.d. as $\mathcal{X}$ and $M_n$ denote a maximizer of  $D_{FT}^n(\cdot; \mathscr{X}).$ 
	Then, 
\begin{equation}\nonumber
		\rho_r(M,M_n), d_r(M,M_n)\longrightarrow 0, \text{ almost surely } [\mathbb{P}], \mbox{ for every } r\in [1,\infty).
\end{equation}

\end{corollary}
	\begin{remark}
		Note that the sequence of sample medians $\{M_n\}_n,$ in Theorem \ref{medianateorema} and Corollary \ref{medianacorolario}, is not necessarily unique.
	\end{remark}

\section{Proofs of the results in Section \ref{main}}\label{pf}

\begin{proof}[Proof of Theorem \ref{TukeyUpperSemi}] 
	Let $(\Omega,\mathcal{A},\mathbb{P})$ be a probabilistic space associated with the fuzzy random variable $\mathcal{X}$. Let $N\subseteq [0,1]$. Set
	$$D_{N}(A;\mathcal{X}) := \inf_{\alpha\in N} D_{CT}(A_{\alpha};\mathcal{X}_{\alpha})$$
	for each $A\in\mathcal{F}_{c}(\mathbb{R}^{p}),$ where $D_{CT}$ is provided in Equation \eqref{CT}.
	 Clearly, $D_{FT}(A;\mathcal{X}) = D_{[0,1]}(A;\mathcal{X})$.
	
	{\it Step 1.}  We will prove $D_{CT}=D_{N}$ for any dense $N\subseteq [0,1]$. Fix $A\in\mathcal{F}_{c}(\mathbb{R}^{p})$ and $\alpha\in [0,1]$.
	
	We first assume $\alpha > 0$. As $N$ is a dense set, there exists an increasing sequence $\{\alpha_{n}\}_{n}\subseteq N$ such that $\alpha_{n} \to \alpha$. Then, we have that $A_{\alpha_{n+1}}\subseteq A_{\alpha_{n}}$ for all $n\in\mathbb{N}$. Thus, $A_\alpha = \bigcap_{n}\downarrow A_{\alpha_{n}}$. If we consider the Hausdorff metric over compact sets, $d_{\mathcal{H}}$, we have that $\lim_{n}d_{\mathcal{H}}(A_{\alpha},A_{\alpha_{n}}) = 0$. Analogously, we have that for all $\omega\in\Omega$, $\mathcal{X}_{\alpha}(\omega) = \bigcap_{n}\downarrow \mathcal{X}_{\alpha_{n}}(\omega)$. Thus, $\lim_{n} d_{\mathcal{H}}(\mathcal{X}_{\alpha}(\omega),\mathcal{X}_{\alpha_{n}}(\omega)) = 0$.
	Since
	\begin{equation}\nonumber
		d_{\mathcal{H}}(K,L) = \sup_{u\in\mathbb{S}^{p-1}} |s_{K}(u) - s_{L}(u)|
	\end{equation}
	for all $K,L\in\mathcal{K}_{c}(\mathbb{R}^{p})$,
	\begin{equation}
		\begin{aligned}\nonumber
			\lim_{n} s_{A_{\alpha_{n}}}(u) &= s_{A_{\alpha}}(u), \text{ for all } u\in\mathbb{S}^{p-1}, \\ \nonumber
			\lim_{n} s_{\mathcal{X}_{\alpha_{n}}(\omega)} &= s_{\mathcal{X}_{\alpha}(\omega)}(u), \text{ for all } u\in\mathbb{S}^{p-1} \text{ and for all } \omega\in\Omega.
		\end{aligned}
	\end{equation}
	Now, for every $u\in\mathbb{S}^{p-1}$ we have the following inequalities:
	\begin{equation}
		\begin{aligned}\nonumber
			\mathbb{P}(s_{\mathcal{X}_{\alpha}}(u)&\geq s_{A_{\alpha}}(u))\geq\mathbb{P}(\lim\sup_{n}\{s_{\mathcal{X}_{\alpha_{n}}}(u)\geq s_{A_{\alpha_{n}}}(u)\})\geq \\ \nonumber
			&\geq\lim\sup_{n}\mathbb{P}(s_{\mathcal{X}_{\alpha_{n}}}(u)\geq s_{A_{\alpha_{n}}}(u)),
		\end{aligned}
	\end{equation}
	where the former inequality is due to the same reason as in the proof of Proposition \ref{propositionTukeyCompact} and the latter due to Fatou's lemma. Taking infima over $u$ in both sides,
	\begin{equation}\label{ecuacion1TukeyFuzzyup}
		\begin{aligned}
			\inf_{u}\mathbb{P}(s_{\mathcal{X}_{\alpha}}(u)\geq s_{A_{\alpha}}(u))&\geq\inf_{u}\inf_{n}\sup_{k\geq n}\mathbb{P}(s_{\mathcal{X}_{\alpha_{k}}}(u)\geq s_{A_{\alpha_{k}}}(u))\geq \\
			&\geq\inf_{n}\inf_{u}\mathbb{P}(s_{\mathcal{X}_{\alpha_{n}}}(u)\geq s_{A_{\alpha_{n}}}(u)).
		\end{aligned}
	\end{equation}
	Analogously with the probability $\mathbb{P}(s_{\mathcal{X}_{\alpha}}(u)\leq s_{A_{\alpha}}(u))$, we have that for all $u\in\mathbb{S}^{p-1}$
	\begin{equation}\label{ecuacion2TukeyFuzzyup}
		\inf_{u}\mathbb{P}(s_{\mathcal{X}_{\alpha}}(u)\leq s_{A_{\alpha}}(u))\geq\inf_{n}\inf_{u}\mathbb{P}(s_{\mathcal{X}_{\alpha_{n}}}(u)\leq s_{A_{\alpha_{n}}}(u)).
	\end{equation}
	Taking the minimum in \eqref{ecuacion1TukeyFuzzyup} and \eqref{ecuacion2TukeyFuzzyup} we obtain
	\begin{equation}\nonumber
		D_{CT}(A_{\alpha};\mathcal{X}_{\alpha})\geq\inf_{n} D_{CT}(A_{\alpha_{n}};\mathcal{X}_{\alpha_{n}})\geq D_{N}(A;\mathcal{X}),
	\end{equation}
	where the second inequality is due to the fact that the $\alpha_{n}$ are in $N$.
	
	At this point, we consider the remaining case $\alpha=0$. By the density of $N$, there exists a {\em decreasing} sequence $\{\alpha_{n}\}_{n}\subseteq N$ such that $\alpha_{n} \to  \alpha$. Then $d_{\mathcal{H}}(A_{0},A_{\alpha_{n}}) = 0$. Following the same argument as for $\alpha>0$, we obtain $D_{CT}(A_{0};\mathcal{X}_{0})\geq D_{N}(A;\mathcal{X}).$
	Thus,
	$$D_{FT}(A;\mathcal X))=\inf_{\alpha\in[0,1]}D_{CT}(A_{\alpha};\mathcal{X}_{\alpha})\geq  D_{N}(A;\mathcal{X})\ge D_F(A;\mathcal X))$$
	since $N\subseteq[0,1]$. By the arbitrariness of $A$, indeed $D_{FT}=D_N$.

	{\it Step 2.} 		
	Let $A\in\mathcal{F}_{c}(\mathbb{R}^{p})$ and let $\{A_{n}\}_{n}$ be a sequence such that $\lim_{n}d_{1}(A,A_{n}) = 0$. We have to prove $\limsup_{n} D_{FT}(A_{n};\mathcal{X})\geq D_{FT}(A;\mathcal{X})$.
	
	Take a subsequence $\{A_{n'}\}_{n}$ such that
	\begin{equation}\nonumber
		\lim_{n} D_{FT}(A_{n'};\mathcal{X}) = \limsup_{n} D_{FT}(A_{n};\mathcal{X}).
	\end{equation}
	Consider the maps
	$$f_{n} : \alpha\in[0,1]\mapsto d_{\mathcal{H}}(A_{\alpha},(A_{n'})_\alpha)\in\mathbb{R}.$$
	Notice $\|f_{n_{k}}\|_1 \to 0$. Since $f_{n_k}$ converges to 0 in probability as random variables defined on $[0,1]$ with the Lebesgue measure $\lambda$, we take a further subsequence $\{n''\}_{n}$ such that the convergence is almost sure.
	
	Set $N = \{\alpha\in [0,1] : f_{n''}(\alpha)\rightarrow 0\}$. Since $N$ has full measure, it is dense in $[0,1]$. By Proposition \ref{propositionTukeyCompact}, for each $\alpha\in N$
	\begin{equation}\nonumber
		D_{CT}(A_{\alpha};\mathcal{X}_{\alpha})\geq\limsup_{n}D_{CT}(A_{n'\alpha};\mathcal{X}_{\alpha}),
	\end{equation}
	because $\lim_{n} f_{n''}(\alpha) = 0$ means $\lim_{n} d_{\mathcal{H}}(A,A_{n''}) = 0$. Thus
	$$				D_{FT}(A;\mathcal{X}) = D_{N}(A;\mathcal{X})\geq\inf_{\alpha\in N}\limsup_{n} D_{CT}((A_{n''})_\alpha;\mathcal{X}_{\alpha})$$
	$$= \inf_{\alpha\in N}\inf_{k\in\mathbb{N}}\sup_{n\geq k} D_{CT}(A_{n''})_\alpha;\mathcal{X}_{\alpha})$$
	$$\geq\inf_{k\in\mathbb{N}}\sup_{n\geq k}\inf_{\alpha\in N} D_{CT}(A_{n''})_\alpha;\mathcal{X}_{\alpha}) $$
	$$=\limsup_{n} D_{N}(A_{n''};\mathcal{X}) = \limsup_{n} D_{FT}(A_{n''};\mathcal{X}) $$
	$$=\lim_{n} D_{FT}(A_{n''};\mathcal{X}) =\limsup_{n} D_{FT}(A_{n};\mathcal{X}),$$
	where the first identity is due to Step 1.
\end{proof}

 \begin{proof}[Proof of Corollary \ref{corolarioupper}]
		The metrics $d_1$ and $\rho_1$ are equivalent (see \cite[Theorem 3.]{diamondkloden}). Thus, thanks to Theorem \ref{TukeyUpperSemi}, $D_{FT}(\cdot;\mathcal X)$ is upper semicontinuous with respect to $\rho_1$.
		
		Given any $r\in (1,\infty)$, $A\in\mathcal F_c(\mathbb R^p)$ a fuzzy set and $\{A_n\}_n$ a sequence of fuzzy sets such that $\lim_n\rho_r(A,A_n) = 0$. Observe that $\rho_1\leq\rho_r$ for every $r\in (1,\infty)$. It implies that $\lim_n\rho_1(A,A_n) = 0$. As $D_{FT}(\cdot ;\mathcal X)$ is upper semicontinuous with respect to the $\rho_1$ metric, we conclude that it is also upper semicontinuous with respect to the $\rho_r$ metric for any $r\in (1,\infty)$.
		
		Also by \cite[Theorem 3.]{diamondkloden}, the metrics $d_r$ and $\rho_r$ are equivalent, thus $D_{FT}(\cdot;\mathcal X)$ is upper semicontinuous with respect to the $d_r$ metric, for every $r\in (1,\infty)$.
		
		Finaly, observe that $d_1\leq d_\infty$. Given any fuzzy set $A\in\mathcal F_c(\mathbb R^p)$ and any sequence $\{A_n\}_n$ of fuzzy sets such that $\lim_nd_\infty(A,A_n) = 0$, we have that $\lim_nd_1(A,A_n) = 0$. Thus, by Theorem \ref{TukeyUpperSemi}, we have that $D_{FT}(\cdot;\mathcal X)$ is upper semicontinuous for the $d_\infty$ metric.
\end{proof}

\begin{proof}[Proof of Theorem \ref{fuzzyTukeyconsistency}]
	
	{\it Step 1.}
	Let $\mathcal X$ be a fuzzy random variable. By measurability results, if $\mathcal X_1,\ldots ,\mathcal X_n$ is a simple random sample of $\mathcal X$, we have that $s_{\mathcal X_1}(u,\alpha), \ldots, s_{\mathcal X_n}(u,\alpha)$ is a random sample of the real random variable $s_{\mathcal X}(u,\alpha)$ for all $u\in\mathbb S^{p-1}$ and $\alpha\in [0,1]$. Let us consider $A\in\mathcal F_c(\mathbb R^p)$. We use the following notation throughout all the proof
	\begin{equation}\nonumber
		\begin{aligned}
			&F(s_A (u,\alpha)) := \{\mathbb P(s_\mathcal X(u,\alpha)\leq s_A(u,\alpha)),\mathbb P(s_\mathcal X(u,\alpha)\geq s_A(u,\alpha)) \}\\ \nonumber
			&F_n(s_A(u,\alpha)) := \{\mathbb P_{u,\alpha}^n((-\infty,s_A(u,\alpha)]),\mathbb P_{u,\alpha}^n([s_A(u,\alpha),\infty)) \}
		\end{aligned}
	\end{equation}
	
	Using Equations \eqref{ecuacion1Tukey} and \eqref{ecuacion2Tukey} and properties about supremum and infimum functions, we have
	\begin{equation}\nonumber
		\begin{aligned}
			&|D_{FT}(A;\mathcal X) - D_{FT}^n(A; \mathscr{X})| = \\ &\left|\inf_{u\in\mathbb{S}^{p-1}}\inf_{\alpha\in [0,1]} \min F(s_A(u,\alpha)) - \inf_{u\in\mathbb{S}^{p-1}}\inf_{\alpha\in [0,1]} \min F_n(s_A(u,\alpha))\right|\leq \\
			&\sup_{u\in\mathbb{S}^{p-1}}\sup_{\alpha\in [0,1]} \left|\min F(s_A(u,\alpha)) - \min F_n(s_A(u,\alpha))\right|
		\end{aligned}
	\end{equation}
	
	For simplicity, we use the following notation
	\begin{equation}\nonumber
		\begin{aligned}
			&F^+(t,u,\alpha) = \mathbb P(s_{\mathcal X}(u,\alpha)\leq t)\\
			&F^-(t,u,\alpha) = \mathbb P(s_{\mathcal X}(u,\alpha)\geq t)\\
			&F_n^+(t,u,\alpha) = \mathbb P_{u,\alpha}^n((-\infty,t])\\
			&F_n^-(t,u,\alpha) = \mathbb P_{u,\alpha}^n([t,\infty)).
		\end{aligned}
	\end{equation}
	
	Applying again properties of the supremum and infimum functions, we obtain
	\begin{equation}\nonumber
		\begin{aligned}
			&|D_{FT}(A;\mathcal X) - D_{FT}^n(A; \mathscr{X})| = \\
			&\sup_{u\in\mathbb{S}^{p-1}}\sup_{\alpha\in [0,1]} \max\{\left|F^+(s_A(u,\alpha),u,\alpha) - F_n^+(s_A(u,\alpha),u,\alpha)\right|,\\
			&\left|F^-(s_A(u,\alpha),u,\alpha) - F_n^-(s_A(u,\alpha),u,\alpha)\right|\}.
		\end{aligned}
	\end{equation}
	Thus,
	\begin{equation}\nonumber
		\begin{aligned}
			&\sup_{A\in\mathcal F_c(\mathbb R^p)} |D_{FT}(A;\mathcal X) - D_{FT}^n(A; \mathscr{X})|\leq\\
			&\sup_{A\in\mathcal F_c(\mathbb R^p)} \sup_{u\in\mathbb{S}^{p-1}}\sup_{\alpha\in [0,1]} \max\{\left|F^+(s_A(u,\alpha),u,\alpha) - F_n^+(s_A(u,\alpha),u,\alpha)\right|,\\
			&\left|F^-(s_A(u,\alpha),u,\alpha) - F_n^-(s_A(u,\alpha),u,\alpha)\right|\}\leq\\
			&\sup_{u\in\mathbb{S}^{p-1}}\sup_{\alpha\in [0,1]}\sup_{t\in\mathbb R}\max\{|F^+(t,u,\alpha)-F_n^+(t,u,\alpha)|,|F^-(t,u,\alpha)-F_n^-(t,u,\alpha)|\}.
		\end{aligned}
	\end{equation}

	{\it Step 2.}
	In Steps 2--5, we will focus on showing
	\begin{equation}\label{prob}
		\sup_{u\in\mathbb{S}^{p-1}}\sup_{\alpha\in [0,1]}\sup_{t\in\mathbb R}|F^+(t,u,\alpha)-F_n^+(t,u,\alpha)|\longrightarrow 0
	\end{equation}
	in probability.
	
	Reasoning by contradiction, assume \eqref{prob} fails. Since the Ky Fan metric metrizes convergence in probability, there would exist some $\eps>0$, some sequence $\{n'\}_n$ of natural numbers, some $\{u_{n'}\}_n\subseteq\mathbb S^{p-1}$, and some $\{\alpha_{n'}\}_n\subseteq[0,1]$ such that, for all $n\in\N$,
	\begin{equation}\label{contrad}
		P\left( \sup_t |F^+(t,u_{n'},\alpha_{n'})-F_{n'}^+(t,u_{n'},\alpha_{n'})|>\eps\right) >\eps.
	\end{equation}
	We start now preparations for Step 3. Since $\|\mathcal X_0\|$ is a random variable, there exists $R>0$ for which $P(\|\mathcal X_0\|>R)<9^{-1}\eps$.
	
	Moreover, since $\mathbb S^{p-1}$ is compact, there exists a subsequence $\{n''\}_n$ and some $u\in \mathbb S^{p-1}$ such that $u_{n''}\to u$.
	
	Finally, since $\{\alpha_{n''}\}_n$ is contained in $[0,1]$ it must have a monotone subsequence $\{\alpha_{n'''}\}_n$ converging to some $\alpha\in[0,1]$. Notice that, for any arbitrary $A\in\pfc(\R^p)$, the sequence $\{A_{\alpha_{n'''}}\}_n$ satisfies
	$$\bigcap_n A_{\alpha_{n'''}}=A_\alpha\mbox{ if }\{\alpha_{n'''}\}_n\mbox{ is increasing, and}$$ $$\bigcup_n A_{\alpha_{n'''}}=\{x\in\R^p\mid A(x)>\alpha\}\mbox{ if }\{\alpha_{n'''}\}_n\mbox{ is decreasing.}$$ Besides, a monotonic sequence of compact sets converging (in the set-theoretical sense) to a relatively compact set in $\R^p$, does converge in the Hausdorff metric to the closure of the set-theoretical limit (see, e.g., \cite[Proposition D.4 and Corollary D.7]{Mol}). Notice the strict level sets $\{x\in\R^p\mid A(x)>\alpha\}$ are indeed relatively compact since their superset $A_0$ is compact.
	
	Since the argument is identical in the increasing and decreasing cases, for ease of notation and without loss of generality we assume $\{\alpha_{n'''}\}_n$ is increasing from now on. From the former paragraph we have, for each $\omega\in\Omega$, the convergence $$d_H(\mathcal X_{\alpha_{n'''}}(\omega),\mathcal X_{\alpha}(\omega))\to 0,\mbox{ whence }\mathcal X_{\alpha_{n'''}}\to \mathcal X_{\alpha}\mbox{ almost surely}$$ in the complete separable metric space $(\pkc(\R^p),d_H)$. The Egorov theorem for metric spaces \cite[Proposition 3.3.11, pp. 105--106]{Par} shows that the convergence is indeed uniform out of a set having positive but arbitrarily small probability.
	
	{\it Step 3.}
	We will now take a suitable auxiliary $\del>0$ which will be instrumental in finding a sufficiently large $n$ for which \eqref{contrad} is violated. That value needs to satisfy several requirements:
	\begin{itemize}
		\item[\namedlabel{(1)}{(a)}] 
		There exists some subset of $\Omega$ with probability at least $1-9^{-1}\eps$, within which\linebreak
		$d_H(\mathcal X_{\alpha_{n'''}},\mathcal X_\alpha)<\del$ for all $n$ larger than some $n_1\in\N$ (independent of $\omega$).
		\item[\namedlabel{(2)}{(b)}] 
		$\|u_{n'''}-u\|\le R^{-1}\del$ for all $n$ larger than some $n_2\in\N$.
		\item[\namedlabel{(3)}{(c)}]  
		$\sup_{t\in\R}\left(F^+(t+2\del,\alpha,u)-F^+(t,\alpha,u)\right)\le 9^{-1}\eps$.
	\end{itemize}
	It was established in Step 2 that all sufficiently small $\del>0$ satisfy the first two requirements. As to the third one, we reason as follows. Rewrite
	$$F^+(t+2\del,\alpha,u)-F^+(t,\alpha,u)=F_{s_{\mathcal X}(u,\alpha)-2\del}(t)-F_{s_{\mathcal X}(u,\alpha)}(t).$$
	Letting $F$ denote the cumulative distribution function of $s_{\mathcal X}(u,\alpha)$ and $F_k$ that of $s_{\mathcal X}(u,\alpha)-2k^{-1}$, to show that Requirement \ref{(3)} holds for some $\del=k^{-1}$ and all smaller $\del$, it will suffice to show $F_k\to F$ uniformly.
	Since $F$ is continuous by the assumption that $X\in C^0[\pfc(\R^p)]$, c.d.f. uniform convergence is implied by pointwise convergence (e.g., \cite[Proposition A.11]{GhoVan}), which in turn follows from
	$$F_k(t)-F(t)=F(t+2k^{-1})-F(t)\to 0.$$
	
	Therefore, we are able to find some $\del>0$ which is sufficiently small to fulfil all three requirements (with the associated $n_1,n_2$).

	{\it Step 4.}
	With the triangle inequality and the properties of suprema,
	\begin{equation}\label{marimar}
		\sup_t |F^+(t,u_{n'''},\alpha_{n'''})-F_{n'''}^+(t,u_{n'''},\alpha_{n'''})|
		\le \sup_t A_{n,t} +\sup_t B_{n,t} +\sup_t C_{n,t}
	\end{equation}
	where
	$$A_{n,t} :=|F^+(t,u_{n'''},\alpha_{n'''})-F^+(t,u,\alpha)|,$$
	$$B_{n,t}:=|F^+(t,u,\alpha)-F_{n'''}^+(t,u,\alpha)|,$$
	$$C_{n,t}:=|F_{n'''}^+(t,u,\alpha)-F_{n'''}^+(t,u_{n'''},\alpha_{n'''})|.$$
	To control $A_{n,t}$, notice that by elementary calculations
	\begin{align}\nonumber\label{At}
		A_{n,t}\le \max\{&P(s_{\mathcal X}(u,\alpha)\le t+|s_{\mathcal X}(u_{n'''},\alpha_{n'''})-s_{\mathcal X}(u,\alpha)|)-P(s_{\mathcal X}(u,\alpha)\le t),\\
		&P(s_{\mathcal X}(u,\alpha)\le t)-P(s_{\mathcal X}(u,\alpha)\le t-|s_{{\mathcal X}(u_{n'''},\alpha_{n'''}})-s_{\mathcal X}(u,\alpha)|\}.
	\end{align}
	Taking into account the definition of $\del$,
	$$|s_{\mathcal X}(u_{n'''},\alpha_{n'''})-s_{\mathcal X}(u,\alpha)|\le |s_{\mathcal X}(u_{n'''},\alpha_{n'''})-s_{\mathcal X}(u_{n'''},\alpha)|+|s_{\mathcal X}(u_{n'''},\alpha)-s_{\mathcal X}(u,\alpha)|$$
	where the first summand is majorized by
	$$\sup_{v\in\mathbb S^{p-1}} |s_{\mathcal X}(v,\alpha_{n'''})-s_{\mathcal X}(v,\alpha)|=d_H(\mathcal X_{\alpha_{n'''}},\mathcal X_\alpha)\le\del$$
	with probability $1-9^{-1}\eps$ (valid for all $n\ge n_1$);
	and the second summand is majorized by
	$$\|u_{n'''}-u\|\cdot \|\mathcal X_\alpha\|\le \|u_{n'''}-u\|\cdot \|\mathcal X_0\|\le R^{-1}\del R=\del$$
	with probability $1-9^{-1}\eps$ (valid for all $n\ge n_2$).
	
	Therefore
	$$|s_{\mathcal X}(u_{n'''},\alpha_{n'''})-s_{\mathcal X}(u,\alpha)|\le 2\del$$
	with probability $1-2\cdot 9^{-1}\eps$ (valid for all $n\ge\max\{n_1,n_2\}$).
	There follows now from \eqref{At}
	$$A_{n,t}\le \max\{F^+(t+2\del,u,\alpha)-F^+(t,u,\alpha)+2\cdot 9^{-1}\eps, F^+(t,u,\alpha)-F^+(t-2\del,u,\alpha)+2\cdot 9^{-1}\eps\}.$$
	Since the first argument of $F^+$ in the expressions $F^+(t+2\del,u,\alpha)-F^+(t,u,\alpha)$ and $F^+(t,u,\alpha)-F^+(t-2\del,u,\alpha)$ is just displaced by $2\del$ units, both have the same supremum over $t\in\R$. And accordingly,
	$$\sup_t A_{n,t}\le \sup_t \left(F^+(t+2\del,u,\alpha)-F^+(t,u,\alpha)\right)+2\cdot 9^{-1}\eps\le 3\cdot 9^{-1}\eps$$
	by Requirement \ref{(3)} on $\delta$. Let us recall explicitly the probability that this inequality holds:
	\begin{equation}\label{111}
		P(\sup_t A_{n,t}\le  3\cdot 9^{-1}\eps)\ge 1-2\cdot 9^{-1}\eps.
	\end{equation}
	
	{\it Step 5.}
	To control $C_{n,t}$ we repeat the reasoning in Step 4 for the empirical distribution function $F_{n'''}$. This requires replacing the distribution $P_{s_{\mathcal X}(u,\alpha)}$ by its empirical distribution $P_{u,\alpha}^n$ but the reasoning is the same until the very end where Requirement \ref{(3)} is invoked. At that point, we only obtain
	$$\sup_t C_{n,t}\le  \sup_t \left(F_{n'''}^+(t+2\del,u,\alpha)-F_{n'''}^+(t,u,\alpha)\right)+2\cdot 9^{-1}\eps.$$
	But since, by the triangle inequality,
	$$F_{n'''}^+(t+2\del,u,\alpha)-F_{n'''}^+(t,u,\alpha)\le B_{n,t+2\del}+(F^+(t+2\del,u,\alpha)-F^+(t,u,\alpha))+B_{n,t},$$
	by taking suprema
	$$\sup_t C_{n,t}\le 3\cdot 9^{-1}\eps+2\sup_t B_{n,t}$$
	still holds, in {\em the same} subset of $\Omega$ having probability at least $1-2\cdot 9^{-1}\eps$ and for the same $n\ge\max\{n_1,n_2\}$.
	
	Accordingly, with \eqref{marimar} and \eqref{111},
	$$P\left(\sup_t |F^+(t,u_{n'''},\alpha_{n'''})-F_{n'''}^+(t,u_{n'''},\alpha_{n'''})|\le 6\cdot 9^{-1}\eps+3\sup_t B_{n,t}\right)\ge 1-2\cdot 9^{-1}\eps.$$
	Almost sure convergence $\sup_t B_{n,t}\to 0$ holds by the Glivenko--Cantelli theorem. Applying Egorov's theorem, with probability at least $1-9^{-1}\eps$ there holds $\sup_t B_{n,t}\le 9^{-1}\eps$ for all sufficiently large $n$, say $n\ge n_3$. We deduce
	$$P\left(\sup_t |F^+(t,u_{n'''},\alpha_{n'''})-F_{n'''}^+(t,u_{n'''},\alpha_{n'''})|\le 9\cdot 9^{-1}\eps\right)\ge 1-3\cdot 9^{-1}\eps$$
	or
	$$P\left(\sup_t |F^+(t,u_{n'''},\alpha_{n'''})-F_{n'''}^+(t,u_{n'''},\alpha_{n'''})|>\eps\right)\le 3^{-1}\eps.$$
	Since this is valid for all $n\ge\max\{n_1,n_2,n_3\}$ a contradiction to \eqref{contrad} has been reached. That establishes \eqref{prob} with convergence in probability.

	{\it Step 6.}
	To handle $F^-$, the argument is identical to that in Steps 2--5. Notice $F^-(t,u,\alpha)$ is just the cumulative distribution function of $-s_{\mathcal X}(u,\alpha)$ evaluated at $-t$ rather than that of $s_{\mathcal X}(u,\alpha)$ at $t$ , whence repeating the argument yields
	$$\sup_{u\in\mathbb{S}^{p-1}}\sup_{\alpha\in [0,1]}\sup_{t\in\mathbb R}|F^-(t,u,\alpha)-F_n^-(t,u,\alpha)|\to 0$$
	in probability. There follows
	\begin{equation}\label{prob2}
		\sup_{u\in\mathbb{S}^{p-1}}\sup_{\alpha\in [0,1]}\sup_{t\in\mathbb R}\max\{|F^+(t,u,\alpha)-F_n^+(t,u,\alpha)|,|F^-(t,u,\alpha)-F_n^-(t,u,\alpha)|\}\to 0
	\end{equation}
	in probability.

	{\it Step 7.}
	In order to prove that almost sure convergence follows, we rewrite \eqref{prob2} to profit from empirical process theory. Setting
	\begin{equation}\nonumber
		\mathcal F = \{\phi_{t,u,\alpha}^+,\phi_{t,u,\alpha}^- : (t,u,\alpha)\in\mathbb R\times\mathbb S^{p-1}\times [0,1]\},
	\end{equation}
	where $\phi_{t,u,\alpha}^+,\phi_{t,u,\alpha}^-:\Omega\rightarrow\mathbb R$ are given by
	\begin{equation}\nonumber
		\begin{aligned}
			&\phi_{t,u,\alpha}^+ (\omega) = \text{I}_{(-\infty,t]}(s_{\mathcal X(\omega)}(u,\alpha))\\ \nonumber
			&\phi_{t,u,\alpha}^- (\omega) = \text{I}_{[t,\infty)}(s_{\mathcal X(\omega)}(u,\alpha)),
		\end{aligned}
	\end{equation}
	we have 	
	\begin{equation}\nonumber
		\begin{aligned}
			&\sup_{u\in\mathbb{S}^{p-1}}\sup_{\alpha\in [0,1]}\sup_{t\in\mathbb R}\{|F^+(t,u,\alpha)-F_n^+(t,u,\alpha)|,|F^-(t,u,\alpha)-F_n^-(t,u,\alpha)|\} \\
			&=\sup_{\phi\in\mathcal F} |\text{E}_{\mathbb P_n}(\phi)-\text{E}_{\mathbb P}(\phi)|
		\end{aligned}
	\end{equation}
	By \cite[Corollary 3.7.9]{gineempiricalprocess}, the supremum converges to $0$ almost surely because it converges in probability and the family $\mathcal F$ has a $\mathbb P$-integrable measurable envelope (since all the functions in $\mathcal F$ take on values in $[0,1]$).
	
	By Step 1, that supremum majorizes $\sup_A |D_{FT}(A;\mathcal X) - D_{FT}^n(A; \mathscr{X})|$ which therefore converges almost surely to 0.
\end{proof}

\begin{proof}[Proof of Theorem \ref{medianateorema}]
	Let $\mathcal X$ be a fuzzy random variable and let $M\in\mathcal F_c(\mathbb R^p)$ the unique maximizer of $D_{FT}(\cdot;\mathcal X)$. Let $\mathcal X_1,\ldots ,\mathcal X_n$ be a simple random sample of $\mathcal X$ and let us denote by $M_n$ a fuzzy random variable that maximizes $D_{FT}^n(\cdot;\mathscr{X})$ for every $n\in\mathbb N$.
	
	{By Corollary \ref{corolarioupper}, the fuzzy Tukey depth $D_{FT}(\cdot;\mathcal X)$ is upper semicontinuous with respect to $d_\infty$, that is, for every $A\in\mathcal F_c(\mathbb R^p)$ and every sequence $\{A_n\}_n$ of fuzzy sets with $\lim_n d_\infty(A,A_n) = 0$, we have that}
	\begin{equation}\nonumber
		\lim\sup_{n\rightarrow\infty}D_{FT}(A;\mathcal X)\leq D_{FT}(A;\mathcal X)
	\end{equation}
	
	By \citet[Theorem 6.6]{primerarticulo}, we have that if $\{A_n\}_n$ is a sequence of fuzzy sets such that $\lim_n d_\infty(A_n,\text{I}_{\{0\}}) = \infty$, then $\lim_n D_{FT}(A_n;\mathcal X) = 0$. As $M$ is the unique maximizer of $D_{FT}(\cdot;\mathcal X)$, for every $\varepsilon > 0$, we have that
	\begin{equation}\nonumber
		D_{FT}(M;\mathcal X) - \sup_{\substack{A \in \mathcal{F}_c(\mathbb{R}^p): \\ d_\infty(A,M) > \varepsilon}}D_{FT}(A;\mathcal X) > 0.
	\end{equation}
	
	We define, for a given $\varepsilon > 0$,
	\begin{equation}\label{delta}
		\delta_\varepsilon := D_{FT}(M;\mathcal X) - \sup_{\substack{A \in \mathcal{F}_c(\mathbb{R}^p): \\ d_\infty(A,M) > \varepsilon}}D_{FT}(A;\mathcal X).
	\end{equation}
	
	Proving that $d_\infty(M,M_n)\longrightarrow 0$ almost surely is sufficient to prove that, given $\varepsilon > 0$, we have that
	\begin{equation}\nonumber
		\mathbb P\left(\left\{\sup_{n > m} d_\infty(M,M_n) > \varepsilon\right\}\right)\longrightarrow 0, \text{ when } m\rightarrow\infty
	\end{equation}
	Thus, let $\varepsilon > 0$ and let $m\in\mathbb N$, we have that
	\begin{equation}\nonumber
		\mathbb P\left(\left\{\sup_{n \geq m}d_\infty(M,M_n) > \varepsilon\right\}\right)\leq\mathbb P\left(\sup_{n\geq m}\left|D_{FT}(M;\mathcal X) - D_{FT}(M_n;\mathcal X)\right|\geq\delta_{\varepsilon}\right),
	\end{equation}
	where the inequality is due to the definition of $\delta_\varepsilon$ and the fact that $M$ is the unique maximizer of $D_{FT}(\cdot;\mathcal X)$ and in \eqref{delta} we take the supremum over all the fuzzy sets with $d_\infty$ distance to $M$ greater than $\varepsilon$.
	
	As $M_n$ is a maximizer of $D_{FT}^n(\cdot;\mathscr{X})$, we have that 
		\begin{equation}\nonumber
		D_{FT}^n(M_n;\mathscr{X}) - D_{FT}^n(M;\mathscr{X}) \geq 0.
			\end{equation}
			 Thus, denoting $R_n:=D_{FT}(M;\mathcal X) + D_{FT}^n(M_n;\mathscr{X}),$
	\begin{equation}\nonumber
		\begin{aligned}
			&\mathbb P\left(\sup_{n\geq m}\left|D_{FT}(M;\mathcal X) - D_{FT}(M_n;\mathcal X)\right|\geq\delta_{\varepsilon}\right)\\
			&\leq\mathbb P\left(\sup_{n\geq m} \left|R_n - D_{FT}^n(M;\mathscr{X}) - D_{FT}(M_n;\mathcal X)\right|\geq\delta_\varepsilon\right)\\
			&\leq\mathbb P\left(\sup_{n\geq m}\left|
			R_n
			\right| + \sup_{n\geq m}\left|D_{FT}(M_n;\mathcal X) - D_{FT}^n(M_n; \mathscr{X})\right|\geq\delta_\varepsilon\right),
		\end{aligned}
	\end{equation}
	where the second inequality is by properties of the supremum. Also by properties of the supremum, for $$T_n:=\sup_{n\geq m}\left|D_{FT}(M;\mathcal X) - D_{FT}^n(M;\mathscr{X})\right| 
			+ \sup_{n\geq m}\left|D_{FT}(M_n;\mathcal X) - D_{FT}^n(M_n;\mathscr{X})\right|,$$
	we have the following inequality
	\begin{equation}\nonumber\displaystyle
		\begin{aligned}
			&\mathbb P\left(T_n\geq\delta_\varepsilon\right)\\
			&\leq\mathbb P\left(\sup_{n\geq m}\left|D_{FT}(M;\mathcal X) - D_{FT}^n(M;\mathscr{X})\right|\geq\delta_\varepsilon/2\right)\\
			&+\mathbb P\left(\sup_{n\geq m}\left|D_{FT}(M_n;\mathcal X) - D_{FT}^n(M_n;\mathscr{X})\right|\geq\delta_\varepsilon/2\right).
		\end{aligned}
	\end{equation}
	
	If we take the supremum over all the fuzzy sets, instead of computing the fuzzy Tukey depth of $M$ and $M_n$, we have the following inequality
	\begin{equation}\label{ecuacionfinal}
		\begin{aligned}
			&\mathbb P\left(\sup_{n\geq m}\left|D_{FT}(M;\mathcal X) - D_{FT}^n(M;\mathscr{X})\right|\geq\delta_\varepsilon/2\right)\\
			&+\mathbb P\left(\sup_{n\geq m}\left|D_{FT}(M_n;\mathcal X) - D_{FT}^n(M_n;\mathscr{X})\right|\geq\delta_\varepsilon/2\right)\\
			&\leq\mathbb P\left(\sup_{n\geq m}\sup_{A\in\mathcal F_c(\mathbb R^p)}\left|D_{FT}(A;\mathcal X)-D_{FT}^n(A;\mathscr{X})\right|\geq\delta_\varepsilon/2\right)\\
			&+\mathbb P\left(\sup_{n\geq m}\sup_{A\in\mathcal F_c(\mathbb R^p)}\left|D_{FT}(A;\mathcal X) - D_{FT}^n(A;\mathscr{X})\right|\geq\delta_\varepsilon/2\right)\\
			&=2\mathbb P\left(\sup_{n\geq m}\sup_{A\in\mathcal F_c(\mathbb R^p)}\left|D_{FT}(A;\mathcal X) - D_{FT}^n(A;\mathscr{X})\right|\geq\delta_\varepsilon/2\right).
		\end{aligned}
	\end{equation}
	
	By Theorem \ref{fuzzyTukeyconsistency}, we have that
	\begin{equation}\nonumber
		\sup_{A\in\mathcal F_c(\mathbb R^p)}\left|D_{FT}(A;\mathcal X) - D_{FT}^n(A;\mathscr{X})\right|\longrightarrow 0, \text{ almost surely } [\mathbb P].
	\end{equation}
	It implies that, for every $\nu > 0$, we have that
	\begin{equation}
		\mathbb P\left(\sup_{n\geq m}\sup_{A\in\mathcal F_c(\mathbb R^p)}\left|D_{FT}(A;\mathcal X) - D_{FT}^n(A;\mathscr{X})\right|\geq\nu\right)\longrightarrow 0, \text{ when } m\rightarrow\infty.
	\end{equation}
	Thus, we have that the last expression in Equation \eqref{ecuacionfinal}
	\begin{equation}\nonumber
		2\mathbb P\left(\sup_{n\geq m}\sup_{A\in\mathcal F_c(\mathbb R^p)}\left|D_{FT}(A;\mathcal X) - D_{FT}^n(A;\mathscr{X})\right|\geq\delta_\varepsilon/2\right)\longrightarrow 0, \text{ when } m\rightarrow\infty.
	\end{equation}
	Then, taking into account all the equations above and bythe Sandwich's Rule, we have that
	\begin{equation}
		\mathbb P\left(\left\{\sup_{n\geq m} d_\infty(M_n,M) > \varepsilon\right\}\right)\longrightarrow 0, \text{ when } m\rightarrow\infty.
	\end{equation}
	Therefore, we have that $d_\infty(M,M_n)\longrightarrow 0$ almost surely, when $n\rightarrow\infty$.
\end{proof}

\section{On the computability of the fuzzy Tukey depth}\label{compt}
	\subsection{Approximated fuzzy Tukey depth}
The fuzzy Tukey depth formulation, in Equations \eqref{TukeyF} and \eqref{ecuacion1Tukey}, includes an infimum over  a no denumerable number of directions $u\in\mathbb S^{p-1}$ for $p>1$ and over a no denumerable number of real values $\alpha\in [0,1].$ This results in that the exact computation of the depth is not possible in general.
As commented in the introduction, the infimum over $\mathbb S^{p-1}$ is inherited from the multivariate Tukey depth  for $p>1$ and, as in that case, the fuzzy Tukey depth can be approximated taking a number of directions from $\mathbb S^{p-1}$ at random; resulting in what we could refer as the {\em random fuzzy Tukey depth}. This is  a fuzzy approximation with very good properties whose proofs follow from those in  \cite{randomTukey, Gabor}.
Thus, our proposal is to follow the same idea the infimum over $\mathbb S^{p-1}.$

Such randomization is also applied in functional spaces \citep{randomTukey}, with the infimum there not necessarily being over $\mathbb S^{p-1}.$ Analogously, a possibility is to also  take values at random in the interval $[0,1],$ to solve the other computational issue. However, we propose here to explore another venue.
%
%
We consider a deterministic finite grid in $[0,1]$, 
\begin{equation}\label{SmG}
S_m = \{\alpha_1,\ldots ,\alpha_n\},
\end{equation} 
and approximate the infimum over $\alpha\in [0,1]$ by the minimum over $\alpha\in S_m.$ 

The univariate fuzzy case ($p=1$) is the most common one in applications. Even in that case, where the unit sphere is finite, the continuum of the interval [0,1] remains. Thus, without loss of generality, we analyze our proposal for $p=1$ in what follows. Next, we define the univariate {\em grid fuzzy Tukey depth} and its sample version making use of Equation \eqref{ecuacion1Tukey}. Analogously, the definitions can be given making use of Equation \eqref{TukeyF}.
\begin{definition}
The grid fuzzy Tukey depth, $D_{FT}^{S_m},$ based on $S_m,$ $\mathcal{H}\subset L^{0}[\mathcal{F}_{c}(\mathbb{R})]$ and $\mathcal{J}\subset\mathcal{F}_{c}(\mathbb{R})$ of a fuzzy set $A\in\mathcal{J}$ with respect to $\mathcal{X}\in\mathcal{H}$ is \begin{equation*}
	\begin{aligned}
	D_{FT}^{S_m}(A;\mathcal X) :=	&\min_{\alpha\in S_m}\min\{HD(s_A(1,\alpha);s_{\mathcal X}(1,\alpha)),HD(s_A(-1,\alpha);s_{\mathcal X}(-1,\alpha))\}.
	\end{aligned}
\end{equation*} 
\end{definition}

\begin{definition}
	The sample grid fuzzy Tukey depth, $D_{FT}^{S_m,n}$, based on $S_m,$ $\mathcal{H}\subset L^{0}[\mathcal{F}_{c}(\mathbb{R})]$  and $\mathcal{J}\subset\mathcal{F}_{c}(\mathbb{R})$ of a fuzzy set $A\in\mathcal{J}$ with respect to  a simple random sample $\mathscr{X}:=\{\mathcal{X}_{1},\ldots ,\mathcal{X}_{n}\}$ from $\mathcal{X}\in\mathcal{H}$
	 is 
	\begin{equation*}
		D_{FT}^{S_m,n}(A;\mathscr{X}) := \min_{\alpha\in S_m}\min\{HD_n(s_A(1,\alpha);s_{n}(1,\alpha)),HD_n(s_A(-1,\alpha);s_{n}(-1,\alpha)) \},
	\end{equation*}
	with 
	\begin{equation}\label{sn}
	s_{n}(u,\alpha):=\{s_{\mathcal{X}_{1}}(u,\alpha),\cdots ,s_{\mathcal{X}_{n}}(u,\alpha)\}.
	\end{equation}
\end{definition}

The question is under what conditions we can assure that
\begin{equation}\label{T1}
	\lim_{n\rightarrow\infty} D_{FT}^{S_m}(A;\mathcal X) = D_{FT}(A;\mathcal X),
\end{equation} and, similarly, a consistent result for $D_{FT}^{S_m,n}$ when $m$ and $n$ go to infinity.
As 
in practice, it is common to consider \textit{triangular fuzzy numbers} \citep[Section 4.1]{librotriangular} in $F_c(\mathbb R)$, we analyze such case in the following theorems. 
Given any real numbers $a\leq b\leq c$, the triangular fuzzy number $T(a,b,c)$ is the fuzzy set given by
\begin{equation}
	\label{difusotriangular}
	\renewcommand{\arraystretch}{2.3} 
	T(a,b,c)(x) := \left\{
	\begin{array}{ll}
		\cfrac{x-a}{b-a}      & \hspace{1cm} \mathrm{if\ } x\in [a,b] \mathrm{\ and\ } a<b\\
		\cfrac{x-c}{b-c} & \hspace{1cm} \mathrm{if\ } x\in [b,c] \mathrm{\ and\ } b<c\\
		1    & \hspace{1cm} \mathrm{if\ } x\in [a,b] \mathrm{\ and\ } a=b \mathrm{\ or\ } x\in [b,c] \mathrm{\ and\ } b = c\\
		0 & \hspace{1cm} \mathrm{otherwise.}
	\end{array}
	\right.
\end{equation}
A triangular fuzzy  random number is a random variable whose realizations are triangular fuzzy  numbers.

Next result provides conditions under which \eqref{T1} is satisfied.
\begin{theorem}\label{teoremaaproximacion}
For any 	 triangular fuzzy  random number $\mathcal X,$  any triangular fuzzy number $A,$ 
and 
\begin{equation}\label{Sm}
S_m := \{i/2^m : i\in\{0,1,\ldots ,2^m\}\},
	\end{equation}
we have that
	$\lim_{m\rightarrow\infty} D_{FT}^{S_m}(A;\mathcal X) = D_{FT}(A;\mathcal X).$
\end{theorem}
The proofs of the results in this section are in the Appendix. 
Next theorem states that the sample approximated fuzzy Tukey depth is strongly consistent with respect to the fuzzy Tukey depth. 
\begin{theorem}\label{teoremaaproximacion2}
	Given the underlying probability space $(\Omega,\mathcal{A},\mathbb{P})$ and $S_m$ as in \eqref{Sm}, let $\mathcal X$ be a triangular fuzzy random number, $\mathscr{X}:=\{\mathcal{X}_{1},\cdots ,\mathcal{X}_{n}\}$ be i.i.d. as $\mathcal{X}$ and $A$ be a triangular fuzzy number. Then, 
	\begin{equation}\nonumber
		|D_{FT}^{S_m,n}(A;\mathscr{X}) - D_{FT}(A;\mathcal X)|\rightarrow 0\text{, almost surely }[\mathbb P],  \text{	as }n,m\rightarrow\infty.
	\end{equation}
\end{theorem}


Next we have results equivalent to Theorem \ref{medianateorema} and Corollary \ref{medianacorolario} for the case of  $D_{FT}^{S_m,n}$ with respect to  triangular fuzzy random numbers. Note they are for a general deterministic grid, as in \eqref{SmG}.
\begin{theorem}\label{teoremamedianasaprox}
	Given the underlying probability space $(\Omega,\mathcal{A},\mathbb{P}),$ let $\mathcal X\in C^0[\mathcal F_c(\mathbb R)]$ be a triangular fuzzy random number such that $D_{FT}(\cdot;\mathcal X)$ is uniquely maximize in $M\in\mathcal F_c(\mathbb R).$ For each $m, n\in \mathbb N,$ let $\mathscr{X}:=\{\mathcal{X}_{1},\cdots ,\mathcal{X}_{n}\}$ be i.i.d. as $\mathcal{X}$ and $M_{m,n}$ denote a maximizer of  $D_{FT}^{S_m,n}(\cdot;\mathscr{X}).$ Then, as $n,m\rightarrow\infty,$
	\begin{equation}\nonumber
		d_\infty(M_{m,n},M) \rightarrow 0\text{, almost surely }[\mathbb P].
	\end{equation}
\end{theorem}

\begin{corollary}\label{Cs}
	Given the underlying probability space $(\Omega,\mathcal{A},\mathbb{P}),$ let $\mathcal X\in C^0[\mathcal F_c(\mathbb R)]$ be a triangular fuzzy random number such that $D_{FT}(\cdot;\mathcal X)$ is uniquely maximize in $M\in\mathcal F_c(\mathbb R).$ For each $m, n\in \mathbb N,$ let $\mathscr{X}:=\{\mathcal{X}_{1},\cdots ,\mathcal{X}_{n}\}$ be i.i.d. as $\mathcal{X}$ and $M_{m,n}$ denote a maximizer of  $D_{FT}^{S_m,n}(\cdot;\mathscr{X}).$ Then, as $n,m\rightarrow\infty,$
\begin{equation}\nonumber
		\rho_r(M_{m,n},M), d_r(M_{m,n},M) \rightarrow 0\text{, almost surely }[\mathbb P], \mbox{ for every } r\in [1,\infty).
	\end{equation}	
\end{corollary}

	\subsection{Simulations}\label{simulations}

To exemplify the previous result, we make use of triangular fuzzy random variables. Following \cite{medianfuzzy1}, we consider them  of the form  
\begin{equation}\label{X}
\mathcal X := T(X-Y,X,X+Z),
\end{equation}where $X$ 
is a continuous real random variable and $Y$ and $Z$ are non-negative random variables; assuring that the triangular fuzzy numbers are well defined.
Taking into account the definition of support function of a fuzzy random variable, in \eqref{srv}, for every $\alpha\in [0,1]$ we have
\begin{equation}\label{s}
	\begin{aligned}
		s_{\mathcal X}(1,\alpha) &= X + (1-\alpha)Z\\
		s_{\mathcal X}(-1,\alpha) &= (1-\alpha)Y - X.
	\end{aligned}
\end{equation}
Then, for any given  triangular fuzzy number $A := T(a,b,c),$ with $a,b,c\in\mathbb R$ such that $a\leq b\leq c$, we have that the fuzzy Tukey depth of $A$ with respect to $\mathcal X$ is
\begin{equation}\nonumber
	\begin{aligned}
		D_{FT}(A;\mathcal X) =& \inf_{\alpha\in [0,1]}\min\{HD(b + (1-\alpha)c;X+(1-\alpha)Z),HD((1-\alpha)a-b;(1-\alpha)Y -X) \}.
	\end{aligned}
\end{equation}
And, by taking into account that the multivariate halfspace depth is invariant under affine transformations \citep{ZuoSerfling}, we can express the fuzzy Tukey depth of $A$ with respect to $\mathcal X$ as
\begin{equation}\nonumber
	\begin{aligned}
		D_{FT}(A;\mathcal X) =& \inf_{\alpha\in [0,1]}\min\{HD(b + (1-\alpha)c;X+(1-\alpha)Z),HD(b- (1-\alpha)a;X-(1-\alpha)Y)\}.
	\end{aligned}
\end{equation}
Analogously
\begin{equation}\nonumber
	\begin{aligned}
	D_{FT}^{S_m}(A;\mathcal X) =	&\min_{\alpha\in S_m}\min\{HD(b+(1-\alpha)c;X+(1-\alpha)Z),HD(b-(1-\alpha)a;X-(1-\alpha)Y)\}.
	\end{aligned}
\end{equation}
and
\begin{equation}\label{smn}
	\begin{aligned}
	D_{FT}^{S_m,n}(A;\mathscr{X}
	)  :=	&\min_{\alpha\in S_m}\min \{HD_n(b+(1-\alpha)c;s_n(1,\alpha)),HD_n(b-(1-\alpha)a;s_n(-1,\alpha)) \},
	\end{aligned}
\end{equation} 
with $s_n(u,\alpha)$ as in \eqref{sn}.
We make use of this expression in the rest of the section to simplify the computations. In the simulations, we take $S_m$ as in \eqref{Sm}.

	\subsubsection{Known population $1$-median}
The aim of this subsection is to assess Corollary \ref{Cs} through a simulation study by observing that, making use of the sample grid fuzzy Tukey depth, the sequence of empirical deepest observed fuzzy set approaches the population $1$-median as the sample size increases.
To this end, let $\mathcal{X}$ be as in \eqref{X}  with $X$ and $Z$ two independent real-valued random variables following a uniform distribution in $(0,1)$ and $Y$ the degenerate distribution that takes always zero value. Note that $\mathcal X\in C^0[\mathcal{F}_c(\mathbb{R})]$  because the real-valued random variables $X$ and $Z$ are continuous.

Following \eqref{s}, the support function of $\mathcal X$ has, for every $\alpha\in [0,1]$, the expression
\begin{equation*}
	\begin{aligned}
		s_{\mathcal X}(1,\alpha) &
	= X + (1-\alpha)Z  \\
	 	s_{\mathcal X}(-1,\alpha) &
	= - X.
	\end{aligned}
\end{equation*}
As $X$ and $Z$ are independent uniform (0,1) distributions, 
$$\text{Med}(X) = 1/2\mbox{ and }\text{Med}(X+(1-\alpha)Z) = 1-\alpha/2$$ for every $\alpha\in [0,1]$. Thus, thanks to \citet[Theorem 4.2]{quintoarticulo} which establishes that the maximizers of the fuzzy Tukey depth are the $1$-medians,
we have that  the $1$-median of the fuzzy random variable $\mathcal X$ is the fuzzy set 
\begin{equation}\label{A}
A := T(1/2,1/2,1).
\end{equation}
$A$ is displayed in green in each of the plots of Figure \ref{figurauniforme}.
 Note that, as $\mathcal{X}$ belongs to $C^0[\mathcal{F}_c(\mathbb{R})]$, its population $1$-median is unique,  due to the uniqueness of the univariate median of continuous distributions and thanks to the existence result in \citet[Theorem 4.1]{medianfuzzy1}. 
		\begin{figure}[!htbp]
	\begin{center} 		\includegraphics[width=0.49\linewidth]{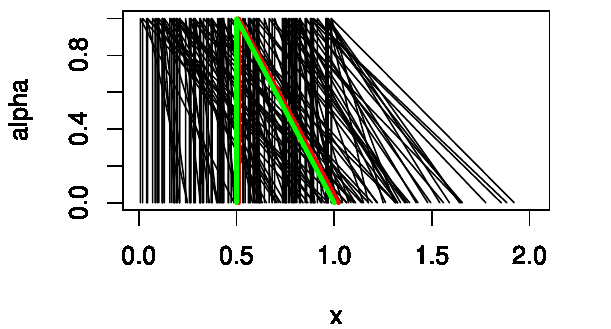} 	
		\includegraphics[width=0.49\linewidth]{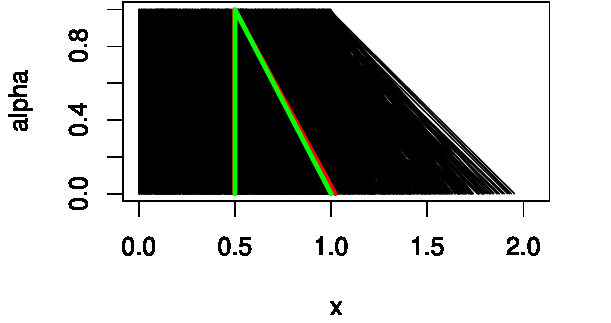} 	
		\includegraphics[width=0.49\linewidth]{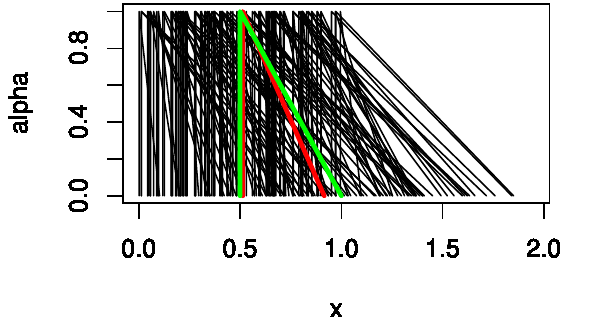} 	
		\includegraphics[width=0.49\linewidth]{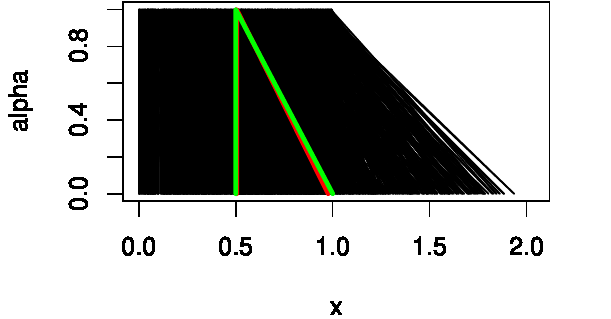} 	
	\end{center} 	\caption{ 	Illustration of the empirical approximated fuzzy Tukey depth. Top row corresponds to the approximation with respect to the set $S_{10}$, while the bottom row corresponds to the approximation with respect to $S_5$. The first column is a sample of $n = 100$ triangular fuzzy sets, while in the second column the sample size is $n = 1000$. Fuzzy sets colored in green corresponds to the $1$-median of $\mathcal X$, while fuzzy sets colored in red corresponds to the fuzzy set in the sample with maximal depth.}	\label{figurauniforme}
\end{figure}


For each $n\in\{100,1000\},$ we draw  a sample $
\mathcal{X}_{1},\ldots ,\mathcal{X}_{n}
$ of i.i.d. random variables, each having the same distribution as  $\mathcal{X}.$  
Figure \ref{figurauniforme} displays those samples. Each of the two samples in the right column plots of  Figure \ref{figurauniforme} are obtained with $n=100$ and each  of the samples in the  left column plots of the figure are  with $n=1000.$ 
Then, for each $n,$ we  select a fuzzy set in the sample that maximizes the empirical grid fuzzy Tukey depth with respect to  the sample, in \eqref{smn}. We do  this for $m \in\{5,10\}.$ The selected fuzzy set is recorded in the last column of Table \ref{T1}  and displayed in red in the plots of Figure \ref{figurauniforme}, in the top row making use of $m =10$ and in the bottom of  $m = 5.$ 
\begin{table*}[!htb]
	\centering
	\begin{tabular}{@{}llcccc@{}}
		\hline
		$n$ & $m$ 
		& \multicolumn{1}{c}{$\rho_2$ distance}
		& \multicolumn{1}{c}{Approximated $1$-median} \\
		
		\hline
		$100$ & $5$ & 0.0335  & $T(0.5137,0.5137,0.9155)$ \\
		$1000$ & $5$    & 0.0094  & $T(0.5047,0.5047,0.9765)$ \\
		$100$ & $10$    & 0.0138 & $T(0.5106,0.5106,1.0217)$ \\
		$1000$ & $10$    & 0.0085  & $T(0.4987,0.4987,1.0214)$ \\
		\hline
	\end{tabular}\caption{ 
	For each $n$ in the first column and each $m$ in the second column, we show in the last column the obtained  $\arg\max_{i=1, \ldots, n}D_{FT}^{S_m,n}(\mathcal{X}_{i};\{\mathcal{X}_{1},\cdots ,\mathcal{X}_{n}\})$ and in the third column the $\rho_2$ distance of this set to $A$ in \eqref{A}.
	As the parameters $n$ and $m$ grow, the $\rho_2$ distance of the approximated $1$-median to the true population $1$-median decreases. 
	}\label{T1}
\end{table*}

We then compare this deepest sample fuzzy sets with $A$ in \eqref{A}, the population $1$-median of $\mathcal{X},$ using the $\rho_2$ distance. The obtained result is shown in the third column of Table \ref{T1} for the different combinations of $n$ and $m.$
%
By Corollary \ref{Cs}, the $\rho_2$ distance between the fuzzy set maximizing the empirical approximated fuzzy Tukey depth and the population $1$-median converges to $0$ as $n \to \infty$. Therefore, we expect this distance to decrease as the sample size increases. 

	\subsubsection{Unknown population $1$-median}\label{Up}
Now, we consider a more general case in which  the true population $1$-median is unknown. 
For that, let 
$\mathcal X$ be as in \eqref{X} with $X, Y$ and $Z$ independent and  continuous real-valued random variables. In particular, let $X$ follow a standard normal distribution 
and $Y$ and $Z$ follow each a chi-squared distribution with $1$ degree of freedom.

Figure \ref{normalfigura} shows different i.i.d. samples $\mathcal{X}_{1},\ldots ,\mathcal{X}_{n}$ with distribution $\mathcal X,$ with size $n=20$ in the left plot, $n=50$ in the middle plot and $n=100$ in the right plot. 
Table \ref{nuevatabla} shows these sample sizes, $n\in\{20, 50, 100\},$ in the first column.
For different values  for the grid $m\in\{1,3,5,10,15\}$, we compute  
\begin{equation}\label{arg}
\arg\max_{i=1, \ldots, n}D_{FT}^{S_m,n}(\mathcal{X}_{i};\{\mathcal{X}_{1},\cdots ,\mathcal{X}_{n}\}),
\end{equation}
 which is the $1$-median approximations of $\mathcal X.$ 
 For the samples illustrated in Figure \ref{normalfigura}, we obtain one unique triangular set in \eqref{arg} for each  $m\in\{1,3,5,10,15\}.$ We display the obtained triangular set in the third column of Table \ref{nuevatabla}, with the corresponding value for $m$ in the second column. We can observe from the table that, for each  $n\in\{20, 50, 100\},$ the same triangular set is depicted for different values of $m.$ We illustrate  in Figure \ref{normalfigura} in blue the triangular set in the first row of each sample size in Table \ref{nuevatabla} and in red the triangular set in the second row of each sample size in the table.

\begin{figure}[!htbp]
	\begin{center}
		\includegraphics[width=0.32\linewidth]{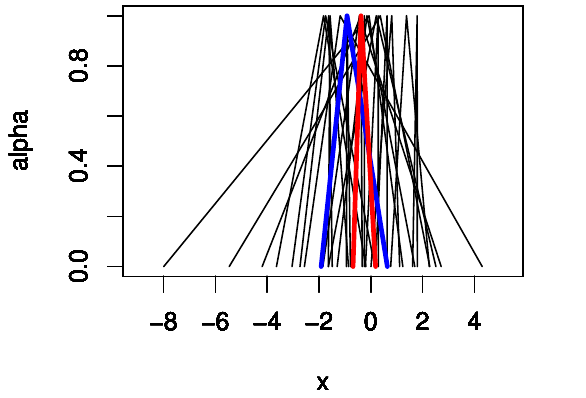}
		\includegraphics[width=0.32\linewidth]{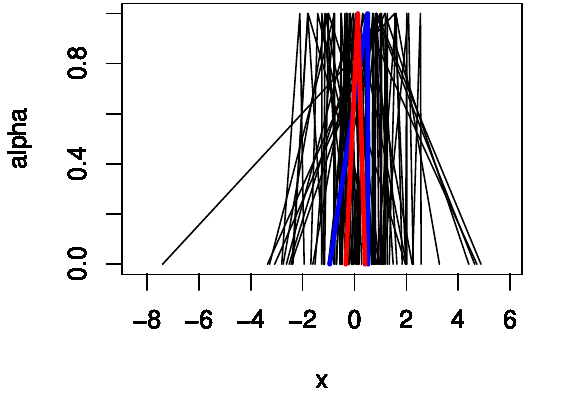}
		\includegraphics[width=0.32\linewidth]{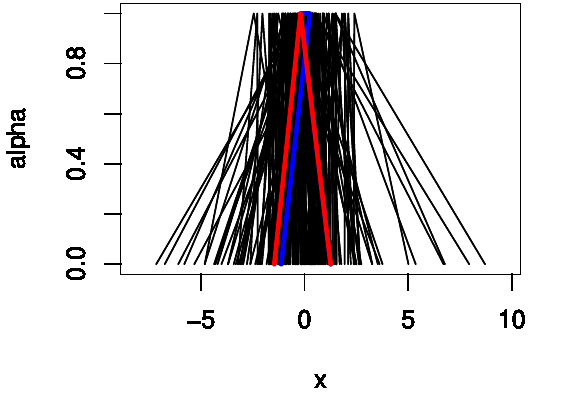}
	\end{center}
	\caption{Empirical approximated fuzzy Tukey depth for samples of sizes $n=20$ (left plot), $n=50$ (middle plot), and $n=100$ (right plot). In each plot, the fuzzy set colored in blue is the sample element with maximal approximated depth, as in \eqref{arg}, for $m\{1, 3\}$ in the left plot, for $m=1$ in the middle plot and for for $m\{3, 5\}$ in the right plot.
	Meanwhile the fuzzy set colored in red is the set resulting from \eqref{arg} which corresponds to $m\{5, 10, 15\}$ in the left plot, to $m\{3, 5, 10, 15\}$ in the middle plot and to $m\{1, 10, 15\}$ in the right plot.}
	\label{normalfigura}
\end{figure}

	\begin{table*}[!htb]
		\centering
		\begin{tabular}{@{}lllllll@{}}
			\hline
			 $n$ & $m$&  &\multicolumn{1}{l}{Approximated $1$-median} \\
			\hline
			$20$ & $1,3$    & & $T(-1.9174,-0.9179,  0.6369)$ \\
			$20$ & $5,10,15$    & & $T(-0.6878,-0.3796, 0.1835)$ \\ \\
			$50$ & $1$ &   &  $T(-0.9538,0.5125,0.5253)$ \\
			$50$ & $3,5,10,15$ &   &  $T(-0.3332,0.1309, 0.4069)$ \\ \\
			$100$ & $3,5$   & &  $T(-1.1313,0.2506,1.1896)$ \\
			$100$ & $1,10,15$  &  &  $T(-1.4633,-0.1988, 1.2567)$ \\
			\hline
		\end{tabular}
		\caption{ When making use of the samples displayed in Figure \ref{normalfigura}, for $n\in\{20, 50, 100\}$ (first column), we display in the third column the triangular set in \eqref{arg}  for each    $m\in\{1,3,5,10,15\}$ (second column). Note that such set is unique for that particular samples and that the same set is obtained for different values of $m.$
		}
		\label{nuevatabla}
	\end{table*}
	
%

Furthermore, in the Appendix we include a simulation in which we show that the same approximated 1-median, \eqref{arg}, is obtained when $m$ increases. There the obtained results for $m\in\{1,3,5,10,15\}$ are compared for those obtained with $m=20$  for different sample sizes $n\in\{20, 50, 100\}.$


			\section{Concluding remarks}\label{conclusions}
			The fuzzy Tukey depth is the first instance of depth introduced in the literature for fuzzy random variables; it coincides with the Tukey, or halfspace, depth when computed on crisp sets and satisfy the two existing notions of fuzzy depth. 
			This work has been dedicated to study its satisfaction of continuity properties,  which are required by a depth function for regularity and inferential stability.
	
	\citet{ZuoSerfling} proved that the halfspace depth is upper-semicontinuous in the first argument with respect to the Euclideand distance and in \citet{segundoarticulo}  we proved that the Tukey depth for compact, convex and non-empty random sets is upper-semicontinuous with respect to the Hausdorff distance (see Proposition \ref{propositionTukeyCompact}). Using this proposition, we have proved in Theorem \ref{TukeyUpperSemi}, and Corollary \ref{corolarioupper}, that the fuzzy Tukey depth is an upper-semicontinuous function with respect to the first argument considering the metrics $d_\infty,$ $d_r$ and $\rho_r,$  for any $r\in [1,\infty).$
	
	We have also proved that the empirical version of the fuzzy Tukey depth is a strong consistent estimator of its population counterpart (see Theorem \ref{fuzzyTukeyconsistency}). This implies that the empirical version of the fuzzy Tukey depht converges uniformly and almost surely to its population version. Thus, Theorem \ref{fuzzyTukeyconsistency} is a Uniform Law of Large Numbers for the empirical process defined by the fuzzy Tukey depth. The proof of this result is based on the proof of the Glivenko-Cantelli Theorem.
	
	Using the upper-semicontinuity and the consistency of the empirical depth we have proved in Theorem \ref{medianateorema}, and Corollary \ref{medianacorolario}, that, if a fuzzy random variable has a unique $1$-median, the distance $d_\infty,$ $d_r$ and $\rho_r,$  for any $r\in [1,\infty),$ between the sequence of maximizers of the empirical version and the $1$-median of the fuzzy random variable converges to $0$ almost surely. This result is based on some ideas proposed by \citet{arcones} for the multivariate case.
	
	The main, and  as far as we are concern only, issue  of the fuzzy Tukey depth is its computation in practice. Even in the case of $p = 1$, to compute the exact value with respect to a fuzzy random variable we have to compute the infimum over $\alpha\in [0,1]$. With the exception of trivial cases, 
	to compute the exact value is unfeasible. Theorem \ref{teoremaaproximacion} provides a way to approximate the exact value of the fuzzy Tukey depth for triangular 
	fuzzy random variables, making use of what we have called the grid fuzzy Tukey depth. The convergence in our result is deterministic. It is also possible to propose a new instance of depth function by considering random selections of $\alpha\in [0,1]$, following the same ideas as in the random Tukey depth proposed by \citet{randomTukey} in the multivariate and functional cases.
	
	Finally, we have proved the convergence of the empirical grid fuzzy Tukey depth in Theorem \ref{teoremaaproximacion2} and of the 
	maximizers of the empirical grid fuzzy Tukey depth in Theorem \ref{teoremamedianasaprox}, and Corollary \ref{Cs}.
	In Section \ref{simulations} we illustrate the convergence of the maximizers. For that we consider two fuzzy random variables. In the first case, we consider a fuzzy random variable for which we know the exact $1$-median, and compare it with the element of the sample that maximizes the empirical grid fuzzy Tukey depth, using the $\rho_2$ distance. In the second case, we consider a more general  fuzzy random variable, for which the $1$-median is unknown, and analyze the convergence of  the fuzzy sets in the sample that maximizes the empirical depth.
	
	For future work, we aim to establish these results for the other existing fuzzy depths in the literature.

			\label{concluding}



\section*{Funding}
The authors are supported by Grant PID2022-139237NB-I00, funded by\linebreak
 MCIN/AEI/10.13039/501100011033 and ``ERDF A way of making Europe''.




\section*{Appendix to Section \ref{compt}}
\subsection{Proofs of the results in Section \ref{compt}}\label{app}
\begin{proof}[Proof of Theorem \ref{teoremaaproximacion}]
	Let us consider the fuzzy random variable $\mathcal X = T(X-Y,X,X+Z)$, where $X$ is a continuous real random variable and $Y$ and $Z$ are two positive and real continuous random variables. Let us consider $A = T(a,b,c)$ a triangular fuzzy number, where $a,b$ and $c$ are real numbers such that $a\leq b\leq c$.
	
	It is clear that $s_{\mathcal X}(1,\alpha) = X+(1-\alpha)Z$ and $s_{\mathcal X}(-1,\alpha) = (1-\alpha)Y-X$ for every $\alpha\in [0,1]$ and $s_A(1,\alpha) = c-\alpha (c-b)$ and $s_A(-1,\alpha) = -a-\alpha(b-a)$ for every $\alpha\in [0,1]$.
	
	If we consider the subsets of $[0,1]$
	
	\begin{equation}\nonumber
		S_m = \{i/2^m : i\in\{0,1,\ldots ,2^m\}\},
	\end{equation}
	
	it is clear that $S_m\subset S_{m+1}$ for every $n\in\mathbb N$. The limit of the sequence $\{S_m\}_m$ is the set
	
	\begin{equation}\label{unionH}
		S = \bigcup_{m=1}^\infty S_m = \lim_{m\rightarrow\infty}S_m,
	\end{equation}
	
	which is a dense subset of $[0,1]$. Using the same argument used in the proof of Theorem \ref{TukeyUpperSemi}, we have that
	
	\begin{equation}\nonumber
		D_{FT}(A;\mathcal X) = \inf_{\alpha\in S}\min\{HD(s_{A}(1,\alpha);s_{\mathcal X}(1,\alpha)),HD(s_{A}(-1,\alpha);s_{\mathcal X}(-1,\alpha))\},
	\end{equation}
	
	where $HD$ is the Tukey depth for the univariate case. Exchanging the infimum with the minimum, we have
	\begin{equation}\nonumber
		D_{FT}(A;\mathcal X) = \min\{\inf_{\alpha\in S}HD(s_{A}(1,\alpha);s_{\mathcal X}(1,\alpha)),\inf_{\alpha\in S}HD(s_{A}(-1,\alpha);s_{\mathcal X}(-1,\alpha))\}.
	\end{equation}
	Without loss of generality, the result is equivalent to prove
	\begin{equation}\nonumber
		\lim_{m\rightarrow\infty}\min_{\alpha\in S_m} HD(s_A(1,\alpha);s_{\mathcal X}(1,\alpha)) = \inf_{\alpha\in S} HD(s_A(1,\alpha);s_{\mathcal X}(1,\alpha))
	\end{equation}
	If we denote by $\mathbb F_X$ the cumulative distribution function of a real random variable $X$, according to the definition of the halfspace depth in the univariate case and following similar arguments than above, the result is equivalent to prove
	\begin{equation}\nonumber
		\lim_{m\rightarrow\infty} \min_{\alpha\in S_m}\mathbb F_{s_\mathcal X(1,\alpha)}(s_A(1,\alpha)) = \inf_{\alpha\in S}\mathbb F_{s_\mathcal X(1,\alpha)}(s_A(1,\alpha))
	\end{equation}
	
	Taking into account the expression of the support function of $\mathcal X$ and $A$, we have to prove
	
	\begin{equation}\nonumber
		\lim_{m\rightarrow\infty} \min_{\alpha\in S_m}\mathbb F_{X+(1-\alpha)Z}(c-\alpha(b-c)) = \inf_{\alpha\in S}\mathbb F_{X+(1-\alpha)Z}(c-\alpha(c-b))
	\end{equation}
	
	Given $\alpha\in S$ and $\{\alpha_m\}_m$ a sequence of elements of $S$ such that $\alpha_m\longrightarrow\alpha$ as $m\rightarrow\infty$, it is clear that $c-\alpha_m(b-c)\longrightarrow c-\alpha(b-c)$ as $m\rightarrow\infty$ and $X+(1-\alpha_m)Z\longrightarrow X+(1-\alpha)Z$ almost surely, as $m\rightarrow\infty$. As the cumulative distribution function of $X+(1-\alpha)Z$ is continuous, we have that the sequence of cumulative distribution functions $\{\mathbb F_{X+(1-\alpha_m)Z}\}_m$ converges uniformly to $\mathbb F_{X+(1-\alpha)Z}$. Thus, the function $f:S\longrightarrow [0,1]$ defined by $f(\alpha) = \mathbb F_{X+(1-\alpha)Z}(c-\alpha(c-b))$ is continuous with respect to $\alpha$. As $S\subset [0,1]$ is a compact set, we have that there exists $\alpha_0\in [0,1]$ such that
	
	\begin{equation}\nonumber
		\inf_{\alpha\in S}\mathbb F_{X+(1-\alpha)Z}(c-\alpha(b-c)) = \mathbb F_{X+(1-\alpha_0)Z}(c-\alpha_0(b-c))
	\end{equation}
	
	As $\lim_{m\rightarrow\infty} S_m = S$, we can find a sequence $\{\alpha_m\}_m$ such that $\alpha_m\in S_m$ for every $m\in\mathbb N$ and $\alpha_m\longrightarrow\alpha_0$ as $m\rightarrow\infty$. Using similar arguments as above, we have that
	\begin{equation}\label{limiteHn}
		\lim_{m\rightarrow\infty}\mathbb F_{X+(1-\alpha_n)Z}(c-\alpha_n(b-c)) = \mathbb F_{X+(1-\alpha_0)Z}(c-\alpha_0(b-c))
	\end{equation}
	
	Thanks to Equation \eqref{unionH}, we have that
	\begin{equation}\nonumber
		\begin{aligned}
			&\mathbb F_{X+(1-\alpha_0)Z}(c-\alpha_0(b-c)) = \inf_{\alpha\in S}\mathbb F_{X+(1-\alpha)Z}(c-\alpha(b-c))\leq\\
			&\min_{\alpha\in S_m}\mathbb F_{X+(1-\alpha)Z}(c-\alpha(b-c))\leq \mathbb F_{X+(1-\alpha_m)Z}(c-\alpha_m(b-c)).
		\end{aligned}
	\end{equation}
	
	Using Equation \eqref{limiteHn} and taking limits in the later expression, we get
	
	\begin{equation}\nonumber
		\lim_{m\rightarrow\infty}\min_{\alpha\in S_m}\mathbb F_{X+(1-\alpha)Z}(c-\alpha(b-c)) = \inf_{\alpha\in S}\mathbb F_{X+(1-\alpha)Z}(c-\alpha(b-c)).
	\end{equation}
	
\end{proof}

\begin{proof}[Proof of Theorem \ref{teoremaaproximacion2}]
	Let $\mathcal X$ be a triangular fuzzy random number, thus $\mathcal X = T(X,Y,Z)$, with $X\leq Y\leq Z$ continuous real-valued random variables. Let $A = T(a,b,c)$ be a triangular fuzzy number, with $a\leq b\leq c$. It is clear that the support function of $\mathcal X$ and $A$ are defined by
	\begin{equation}\nonumber
		\begin{aligned}
			s_\mathcal X(1,\alpha) &= Z+\alpha(Y-Z),\text{ }s_\mathcal X(-1,\alpha) = -(X+\alpha(Y-X)),\\
			s_A(1,\alpha) &= c+\alpha(b-c),\text{ }s_A(-1,\alpha) = -(a+\alpha(b-a)),
		\end{aligned}
	\end{equation}
	
	for every $\alpha\in [0,1]$. We can express $D_{FT}$ and $D_{FT}^{S_m}$ as
	
	\begin{equation}\nonumber
		\begin{aligned}
			D_{FT}(A;\mathcal X) &= \inf_{\alpha\in [0,1]}\min\{HD(c+\alpha(b-c);Z+\alpha(Y-Z)),HD(a+\alpha(b-a);X+\alpha(Y-X))\}\\
			D_{FT}^{S_m}(A;\mathcal X) &= \min_{\alpha\in S_m}\min\{HD(c+\alpha(b-c);Z+\alpha(Y-Z)),HD(a+\alpha(b-a);X+\alpha(Y-X))\},
		\end{aligned}
	\end{equation}
	
	for every $m\in\mathbb N$.
	
	Let $\mathcal X_1,\ldots ,\mathcal X_n$ be a simple random sample of $\mathcal X$. By measurability results, we can assume that we have simple random samples $\{X_i\}_{i=1}^n$, $\{Y_i\}_{i=1}^n$ and $\{Z_i\}_{i=1}^n$ of $X, Y$ and $Z$, respectively and $\mathcal X_i = T(X_i,Y_i,Z_i)$ for every $i = 1,\ldots ,n$. Using the triangle inequality, we have that
	\begin{equation}\nonumber
		\left|D_{FT}^{S_m,n}(A;\mathcal X) - D_{FT}(A;\mathcal X)\right|\leq \left|D_{FT}^{S_m,n}(A;\mathcal X) - D_{FT}^{S_m}(A;\mathcal X)\right| + \left|D_{FT}^{S_m}(A;\mathcal X) - D_{FT}(A;\mathcal X)\right|
	\end{equation}
	
	Using Theorem \ref{teoremaaproximacion}, we have that the second term in the right side of the inequality converges to $0$ as $m\rightarrow\infty$. Thus, to prove our result we need to prove that the first term converges to $0$ almost surely as $n,m\rightarrow\infty$.
	
	Let $\alpha\in [0,1]$ and define $u_\alpha = c+\alpha (b-c), v_\alpha = a+\alpha (b-a), U_\alpha = Z+\alpha(Y-Z)$ and $V_\alpha = X+\alpha(Y-X)$. With this notation, the expression of $D_{FT}^{S_m}(A;\mathcal X)$ becomes
	\begin{equation}\nonumber
		D_{FT}^{S_m}(A;\mathcal X) = \min_{\alpha\in S_m}\min\{HD(u_\alpha;U_\alpha),HD(v_\alpha;V_\alpha)\}
	\end{equation}
	On the other hand, let us define $U_{\alpha,i} = Z_i+\alpha(Y_i-Z_i)$ and $V_{\alpha,i} = X_i+\alpha(Y_i-X_i)$. Define also the functions $g(\alpha) = \min\{HD(u_\alpha;U_\alpha),HD(v_\alpha;V_\alpha)\}$ and $g_n(\alpha)$ as it sample version based on the samples  $\{X_i\}_{i=1}^n$, $\{Y_i\}_{i=1}^n$ and $\{Z_i\}_{i=1}^n$. Thus, we have that
	
	\begin{equation}\nonumber
		\begin{aligned}
			D_{FT}^{S_m}(A;\mathcal X) &= \min_{\alpha\in S_m} g(\alpha)\\
			D_{FT}^{S_m,n}(A;\mathcal X) &= \min_{\alpha\in S_m} g_n(\alpha).
		\end{aligned}
	\end{equation}
	
	With this notation, we have
	
	\begin{equation}\nonumber
		\left|D_{FT}^{S_m,n}(A;\mathcal X) - D_{FT}^{S_m}(A;\mathcal X)\right|\leq\sup_{\alpha\in S_m}|g_n(\alpha) - g(\alpha)|\leq\sup_{\alpha\in [0,1]}|g_n(\alpha) - g(\alpha)|
	\end{equation}
	
	Taking into account the definitions of the functions $g_n$ and $g$ and using properties of the supremum and the infimum, we have the following inequalities
	
\begin{equation}
	\begin{aligned}
		\sup_{\alpha\in [0,1]}  |g_n(\alpha) & -g(\alpha)|
		\leq \\
		\leq&\sup_{\alpha\in [0,1]}
		\max\Bigl\{
		|HD(u_\alpha;U_{\alpha,n})-HD(u_\alpha;U_\alpha)|, 
		|HD(v_\alpha;V_{\alpha,n})-HD(v_\alpha;V_\alpha)|
		\Bigr\} \\
		\leq &
		\sup_{\alpha\in [0,1]}
		\max\Bigl\{
		|\mathbb{P}(U_{\alpha,n}\le u_\alpha)-\mathbb{P}(U_\alpha\le u_\alpha)|, 
		|\mathbb{P}(U_{\alpha,n}\ge u_\alpha)-\mathbb{P}(U_\alpha\ge u_\alpha)|,
		 \\&\qquad\qquad\hspace{.45cm}
		|\mathbb{P}(V_{\alpha,n}\le v_\alpha)-\mathbb{P}(V_\alpha\le v_\alpha)|, 
		|\mathbb{P}(V_{\alpha,n}\ge v_\alpha)-\mathbb{P}(V_\alpha\ge v_\alpha)|
		\Bigr\}.
	\end{aligned}
\end{equation}

It is sufficies to prove that 
\begin{equation}\nonumber
	\sup_{\alpha\in [0,1]}|\mathbb{P}(U_{\alpha,n}\le u_\alpha)-\mathbb{P}(U_\alpha\le u_\alpha)|\rightarrow 0,\text{ a.s.} [\mathbb P].
\end{equation}

The proof for the other terms is analogous. Let us denote by $\mathbb P_n$ the empirical measure of the sample. Then, we have that, if we consider, for any given $\alpha\in [0,1]$, the set 
\begin{equation}\nonumber
	\mathcal C_\alpha = \{(y,z)\in\mathbb R^2 : z+\alpha(y-z)\leq u_\alpha\} = \{(y,z)\in\mathbb R^2 : (z-c) + \alpha ((y-z)-(b-c))\leq 0\},
\end{equation}
we have that
\begin{equation}\nonumber
	\sup_{\alpha\in [0,1]}\left|\mathbb{P}(U_{\alpha,n}\le u_\alpha)-\mathbb{P}(U_\alpha\le u_\alpha)\right| = \sup_{\alpha\in [0,1]}\left|\mathbb P_n(\mathcal C_\alpha) - \mathbb P(\mathcal C_\alpha)\right|,
\end{equation}
for every $\alpha\in [0,1]$. It is easy to see that each $\mathcal C_\alpha$ is a closed halfspace in $\mathbb R^2$. Thus, the class of sets 
\begin{equation}\nonumber
	\mathcal C = \{\mathcal C_\alpha : \alpha\in [0,1]\}
\end{equation}
is contained in the class of closed halfspaces of $\mathbb R^2$, which has finite VC-dimension (see \cite{vanderwaartwellner}). It implies that $\mathcal C$ is a $\mathbb P$-Glivenko-Cantelli class (see \cite{shorackwellner}). As
\begin{equation}\nonumber
	\sup_{\alpha\in [0,1]}\left|\mathbb{P}(U_{\alpha,n}\le u_\alpha)-\mathbb{P}(U_\alpha\le u_\alpha)\right| = \sup_{C\in\mathcal C}\left|\mathbb P_n(C) - \mathbb P(C)\right|,
\end{equation}
we conclude that
\begin{equation}\nonumber
	\sup_{\alpha\in [0,1]}\left|\mathbb{P}(U_{\alpha,n}\le u_\alpha)-\mathbb{P}(U_\alpha\le u_\alpha)\right|\longrightarrow 0,\text{ a.s. }[\mathbb P]
\end{equation}	
\end{proof}

\begin{proof}[Proof of Theorem \ref{teoremamedianasaprox}]
	Let $\mathcal X\in C^0[\mathbb R]$ be a triangular fuzzy random number, thus $\mathcal X = T(X,Y,Z)$ with $X\leq Y\leq Z$ continuous real-valued random variables. Thus, the $\alpha$-level of $\mathcal X$ is 
	\begin{equation}\nonumber
		(\mathcal X)_\alpha = [\alpha Y + (1-\alpha)X, \alpha Y + (1-\alpha)Z],
	\end{equation}
	
	for every $\alpha\in [0,1]$. Let us consider $\mathcal X_1,\ldots ,\mathcal X_n$ a simple random sample of $\mathcal X$. By measurability results, we can assume that we have simple random samples $\{X_i\}_{i=1}^n$, $\{Y_i\}_{i=1}^n$ and $\{Z_i\}_{i=1}^n$ of $X, Y$ and $Z$, respectively and $\mathcal X_i = T(X_i,Y_i,Z_i)$ for every $i = 1,\ldots ,n$. 
	
	Let $\{M_{m,n}\}$ be a sequence of maximizers of $D_{FT}^{S_m,n}(\cdot;\mathcal X)$. For every $\alpha\in [0,1]$, let us denote the $\alpha$-levels 
	\begin{equation}\nonumber
		(M_{m,n})_\alpha = [m^-_{m,n}(\alpha), m^+_{m,n}(\alpha)]
	\end{equation}
	
	Following the same ideas as in \cite[Theorem 4.4]{quintoarticulo}, we have that, for every $n,m\in\mathbb N$, we have that
	
	\begin{equation}\nonumber
		\begin{aligned}
		m^-_{m,n}(\alpha) &= Med_n\left(\alpha Y + (1-\alpha)X\right),\\
		m^+_{m,n}(\alpha) &= Med_n\left(\alpha Y + (1-\alpha)Z\right),\\
	\end{aligned}
	\end{equation}
	
	for every $\alpha\in S_m$, where $Med_n$ denotes the empirical median associated with the samples $\{\alpha Y_i + (1-\alpha)X_i\}_{i = 1}^n$ and $\{\alpha Y_i + (1-\alpha)Z_i\}_{i=1}^n$, respectively. 
	
	As $\mathcal X\in C^0[\mathcal F_c(\mathbb R)]$ and $\mathcal X = T(X,Y,Z)$, we have that there exists a unique $1$-median $M\in\mathcal F_c(\mathbb R)$ of $\mathcal X$. In particular, by \cite[Theorem 4.2]{quintoarticulo}, we have that if $(M)_\alpha = [m^-(\alpha),m^+(\alpha)]$, for every $\alpha\in [0,1]$, then 
	\begin{equation}\nonumber
		\begin{aligned}
		m^-(\alpha) &= Med(\alpha Y+(1-\alpha)X)\\
		m^+(\alpha) &= Med(\alpha Y+(1-\alpha)Z)
	\end{aligned}
	\end{equation}
	
	As $X,Y$ and $Z$ are continuous random variables, we have that $\alpha Y+(1-\alpha)X$ and $\alpha Y+(1-\alpha)Z$ are also continous random variables. In fact, the functions $m^-(\cdot)$ and $m^+(\cdot)$ are continuous functions in $\alpha$. As $[0,1]$ is compact, we have that the functions $m^-(\cdot)$ and $m^+(\cdot)$ are uniformly continuous on $[0,1]$. 
	
	On the other hand, the simple random sample $\mathcal X_1,\ldots ,\mathcal X_n$ can be expressed as\linebreak
	 $\mathcal X_i = T(X_i,Y_i,Z_i)$ for every $i = 1,\ldots ,n$.
	Let us denote for every $\alpha\in [0,1]$
	
	\begin{equation}\nonumber
		\begin{aligned}
		m_n^-(\alpha) &= Med_n(\alpha Y + (1-\alpha)X),\\
		m_n^+(\alpha) &= Med_n(\alpha Y + (1-\alpha)Z).
	\end{aligned}
	\end{equation}
	
	Let us write $U_\alpha = \alpha Y + (1-\alpha)X$ and $U_{\alpha,i} = \alpha Y_i + (1-\alpha)X_i$ for every $\alpha\in [0,1]$ and $i = 1,\ldots ,n$. Let us define $F_\alpha(t) = \mathbb P(U_\alpha\leq t)$ and
	\begin{equation}\nonumber
	 F_{\alpha,n}(t) = \cfrac{1}{n}\sum_{i=1}^n \mathbb I_{\{U_{\alpha,i}\leq t\}}.
	 \end{equation} 
	 
	 Let us consider the closed halfspace $\mathcal C_{\alpha,t} = \{(x,y)\in\mathbb R^2 : \alpha y + (1-\alpha)x\leq t\}$ for every $\alpha\in [0,1], t\in\mathbb R$. It is clear that the class $\mathcal C = \{\mathcal C_{\alpha,t} : \alpha\in [0,1], t\in\mathbb R\}$ is contained in the class of closed halfspaces of $\mathbb R^2$, which is a $\mathbb P$-Glivenko-Cantelli class (see \cite{shorackwellner,vanderwaartwellner}). It implies that 
	 
	 \begin{equation}\label{ecuacionsupremo}
	 	\sup_{\alpha\in [0,1]}\sup_{t\in\mathbb R}|F_{\alpha,n}(t) - F_\alpha(t)|\rightarrow 0,\text{ almost surely }[\mathbb P]
 	 \end{equation}
	 
	 For every $\alpha\in [0,1]$, we have that $F_\alpha(m^-(\alpha)) = 1/2$. As the median of $\alpha Y + (1-\alpha)X$ is unique, we have that, for every $\varepsilon >0$, we have that
	 \begin{equation}\nonumber
	 	F_\alpha(m^-(\alpha)-\varepsilon) < \cfrac{1}{2} < F_\alpha(m^-(\alpha)+\varepsilon).
	\end{equation}
	
	Let us define, for every $\varepsilon >0$ and every $\alpha\in [0,1]$, the function
	
	\begin{equation}
		g_\varepsilon (\alpha) = \min\{1/2 - F_\alpha(m^-(\alpha) - \varepsilon),F_\alpha(m^-(\alpha)+\varepsilon) - 1/2 \} > 0.
	\end{equation}
	
	As the functions $\alpha\rightarrow m^-(\alpha)$ and $(\alpha,t)\rightarrow F_\alpha(t)$ are continous, we have that the function $\alpha\rightarrow g_\varepsilon(\alpha)$ is continuous on $[0,1]$. It implies that
	
	\begin{equation}\nonumber
		c_\varepsilon = \inf_{\alpha\in [0,1]} g_\varepsilon(\alpha) > 0.
	\end{equation}
	
	By Equation \eqref{ecuacionsupremo}, for a sufficiently large $n$, we have that
	
	\begin{equation}\nonumber
		\sup_{\alpha\in [0,1]}\sup_{t\in\mathbb R}|F_{\alpha,n}(t) - F_\alpha (t)|\leq c_\varepsilon,\text{ almost surely }[\mathbb P].
	\end{equation}
	
	Therefore, we have that, for every $\alpha\in [0,1]$, we have that
	\begin{equation}\nonumber
		\begin{aligned}
		&F_{\alpha,n}(m^-(\alpha) - \varepsilon) < \cfrac{1}{2}\\
		&F_{\alpha,n}(m^-(\alpha) + \varepsilon) > \cfrac{1}{2}\\
		\end{aligned}
	\end{equation}
	
	It implies that 
	\begin{equation}\nonumber
		\sup_{\alpha\in [0,1]}|m_n^-(\alpha) - m^-(\alpha)|\leq\varepsilon,\text{ almost surely }[\mathbb P].
\end{equation}
	
	Since $\varepsilon >0$ is taken arbitrarily, we conclude that 
	\begin{equation}\label{convergenciamedianas}
		\sup_{\alpha\in [0,1]}|m_n^-(\alpha) - m^-(\alpha)|\rightarrow 0,\text{ almost surely }[\mathbb P]
	\end{equation}
	
	Analogously, we have that 
	
	\begin{equation}\nonumber
		\sup_{\alpha\in [0,1]}|m_n^+(\alpha) - m^+(\alpha)|\rightarrow 0,\text{ almost surely }[\mathbb P]
	\end{equation}
	
	By hypothesis, given $n,m\in\mathbb N$, we have that
	
	\begin{equation}\label{ecuacioensmedianas}
		\begin{aligned}
			m_{m,n}^-(\beta) &= m_n^-(\beta),\\
			m_{m,n}^+(\beta) &= m_n^+(\beta),
		\end{aligned}
	\end{equation}
	
	for every $\alpha\in S_m$. Let $n,m\in\mathbb N$. Given any $\alpha\in [0,1]$, choose $\beta\in S_m$ such that $|\alpha-\beta|\leq 2^{-m}$. Without loss of generality, let us assume that $\beta\geq\alpha$. Thus,
	\begin{equation}\label{ecuacionprincipal}
		\begin{aligned}
			|m_{m,n}^-(\alpha) &- m^-(\alpha)|\\
			&\leq |m_{m,n}^-(\alpha) - m_{m,n}^-(\beta)| + |m_{m,n}^-(\beta) - m^-(\beta)| + |m^-(\beta)-m^-(\alpha)|\\
			&\leq|m_{m,n}^-(\alpha) - m_{m,n}^-(\beta)| + \sup_{\alpha\in [0,1]}|m_n^-(\alpha) - m^-(\alpha)| + |m^-(\beta)-m^-(\alpha)|\\
			&\leq|m_{m,n}^-(\beta - 2^{-m}) - m_{m,n}^-(\beta)|+ \sup_{\alpha\in [0,1]}|m_n^-(\alpha) - m^-(\alpha)| + |m^-(\beta)-m^-(\alpha)|\\
			&\leq\sup_{\alpha\in [0,1]}|m_n^-(\alpha-2^{-m})-m_n^-(\alpha)|+ \sup_{\alpha\in [0,1]}|m_n^-(\alpha) - m^-(\alpha)| + |m^-(\beta)-m^-(\alpha)|,
		\end{aligned}
	\end{equation}
	where Equation \eqref{ecuacioensmedianas} are used in the second and fourth inequalities. The second converges to $0$ almost surely thanks to Equation \eqref{convergenciamedianas}, while the third term conveges to $0$ as the function $\alpha\rightarrow m^-(\alpha)$ is uniformly continuous. We can bound the first term 
	\begin{equation}\nonumber
		\begin{aligned}
			&\sup_{\alpha\in [0,1]}|m_n^-(\alpha-2^{-m})-m_n^-(\alpha)|\leq\sup_{\alpha\in [0,1]} |m_n^-(\alpha-2^{-m}) - m^-(\alpha - 2^{-m})|\\ 
			&+ \sup_{\alpha\in [0,1]} |m^-(\alpha-2^{-m}) - m^-(\alpha)| +\sup_{\alpha\in [0,1]}|m_n^-(\alpha) - m^-(\alpha)|.
		\end{aligned}
	\end{equation}
	The right side converges to $0$ almost surely thanks to Equation \eqref{convergenciamedianas} and thanks to the fact that the function $\alpha\rightarrow m^-(\alpha)$ is uniformly continuous. 
	
	Taking supremums on both sides we conclude that 
	\begin{equation}\nonumber
		\sup_{\alpha\in [0,1]}|m_{m,n}^-(\alpha) - m^-(\alpha)|\rightarrow 0,\text{ almost surely }[\mathbb P].
	\end{equation}
	Analogously, we also conclude that 
		\begin{equation}\nonumber
		\sup_{\alpha\in [0,1]}|m_{m,n}^+(\alpha) - m^+(\alpha)|\rightarrow 0,\text{ almost surely }[\mathbb P].
	\end{equation}
	
	Taking into account the definition of $d_\infty$ metric, 	as $n,m\rightarrow\infty,$ we have that
	\begin{equation}\nonumber
		d_\infty(M_{m,n},M) = \sup_{\alpha\in [0,1]}\max\{|m_{m,n}^-(\alpha) - m^-(\alpha)|,|m_{m,n}^+(\alpha)-m^+(\alpha)|\}
	\end{equation}
	Thus, with the above results, we conclude that \begin{equation}\nonumber
		d_\infty(M_{m,n},M)\rightarrow 0,\text{ almost surely }[\mathbb P].
	\end{equation}
\end{proof}

\subsection{Additional simulations}

Here we further analyse the unknown population 1-median case of Subection \ref{Up}. For that, we do 100 iterations of the following procedure. In each iteration, we draw  $\mathcal{X}_{1},\ldots ,\mathcal{X}_{100}$ independently with distribution $\mathcal X,$ described in Subection \ref{Up}. We analyze this sample and the subsamples $\mathcal{X}_{1},\ldots ,\mathcal{X}_{50}$ and $\mathcal{X}_{1},\ldots ,\mathcal{X}_{20}.$ We refer to these cases as $n\in\{100, 50, 20\},$ respectively. For each of them, we compute the 1-median approximations in \eqref{arg} making use of   $m\in\{1,3,5,10,15, 20\}.$ Note that, as it happens with the, one-dimensional, median,  the 1-median approximation is not necessarily a unique  fuzzy number. 

Our objective here is to compare, for each $n\in\{20, 50, 100\},$ the fuzzy number(s) obtained with $m=20$ with those obtained with $m\in\{1,3,5,10,15\}.$ Thus, for each iteration, $n\in\{20, 50, 100\}$ and $m\in\{1,3,5,10,15\},$   we check if the obtained fuzzy number(s) coincide with the one(s) obtained with $m=20.$ The amount of coincident fuzzy numbers is divided by the amount of fuzzy numbers obtained  with $m=20.$ We add this value over the 100 iterations and display the results in Table \ref{tablaapp}.
	\begin{table}[!htb]
		\centering
		\begin{tabular}{@{}llllll@{}}
			\hline
			  & $m=1$&$m=3$  &$m=5$ & $m=10$&$m=15$\\
			\hline
			$n=20$ & 96.3&  98.5& 99.5& 100& 100 \\
			$n=50$ & 89.3&  95.5&  99.5&100 &100\\ 
			$n=100$ & 89.8 &  94.2&  98.3& 100& 100 \\
			\hline
		\end{tabular}
		\caption{Over 100 iterations, amount of times that the  1-median approximation for $m=20$ coincides with that obtained for  $m=1$ (second column), $m=3$  (third column), $m=5$ (fourth column),  $m=10$ (fifth column) and $m=15$ (sixth column). This is done for samples with size $n=100$ (fourth row), and the corresponding subsamples of size $n=50$ (third row) and $n=20$ (second row). When two compared 1-median approximations do not have the same amount of elements, we count how many elements coincide and divide it by the amount of elements obtained with $m=20.$ 
		}
		\label{tablaapp}
	\end{table}
	
Note that generally we obtain only one fuzzy number for each  iteration, $n\in\{20, 50, 100\}$ and  $m\in\{1,3,5,10,15, 20\}.$ However, to explain it in more detail, let us assume that we obtain $k\in\mathbb{N}$ fuzzy numbers with $m=20$ and $l\in\mathbb{N}$ for another $m,$ for a fixed $n$ and iteration. Le us assume that $h\leq\min\{k, l\}$ fuzzy numbers coincide among the two sets, then in that particular iteration, to obtain the values in Table \ref{tablaapp} we add the value $h/k.$

The results in   Table \ref{tablaapp} show that the obtained 1-median approximations are stable as $m$ increases. Thus, for each of the 100 iterations, the same 1-median approximations were obtained for  $m\in\{10,15, 20\}.$ For smaller values of $m$ there are very few cases in which the obtained 1-median approximations does not coincide entirely with those obtained for $m=20.$

\end{document}